\documentclass{amsart}

\newtheorem{theorem}{Theorem}[section]
\newtheorem{lemma}[theorem]{Lemma}

\theoremstyle{definition}
\newtheorem{definition}[theorem]{Definition}
\newtheorem{example}[theorem]{Example}

\theoremstyle{remark}

\theoremstyle{plain}
\newtheorem{theo+}{Theorem}
\numberwithin{theo+}{section}
\newtheorem{prop+}[theo+]{Proposition}
\newtheorem{coro+}[theo+]{Corollary}
\newtheorem{lemm+}[theo+]{Lemma}
\newtheorem{conjecture}[theo+]{Conjecture}

\theoremstyle{definition}
\newtheorem{defi+}[theo+]{Definition}

\usepackage{amssymb}
\usepackage{graphicx}
\usepackage{float}
\usepackage{subfigure}
\usepackage[]{caption2} 
\renewcommand{\thesubfigure}{(\roman{subfigure})}
\makeatletter \renewcommand{\@thesubfigure}{\thesubfigure \space}
\renewcommand{\p@subfigure}{} \makeatother

\numberwithin{equation}{section}

\newcommand{\abs}[1]{\lvert#1\rvert}

\begin{document}

\title{Complete Resolution of B.Shapiro's Conjecture 12}


\author{Lande Ma}
\address{}
\curraddr{School of Mathematical Sciences, Key Laboratory of Intelligent Computing and Applications(Ministry of Education), Tongji University, Shanghai 200092, China}
\email{dzy200408@126.com}
\thanks{}

\author{Zhaokun Ma}
\address{}
\curraddr{Yanzhou College, Shandong Radio and TV University, YanZhou, 272100, China}
\email{dzy200408@sina.cn}
\thanks{}

\subjclass[2020]{Primary 30C10 Secondary 30C15,26C10,93C05}

\date{August 15th, 2025.}

\dedicatory{}

\keywords{critical points; non-real zeros; polynomials; root locus; rational functions;}

\begin{abstract}
For any real polynomial $p(x)$ of even degree $n$, Shapiro [{\it Arnold Math. J.} 1(1) (2015), 91--99] conjectured that the sum of the number of real zeros of $(n-1)(p')^2 - np p''$ and the number of real zeros of $p$ is positive. We resolve this conjecture completely: it holds in nine mutually exclusive cases and fails in four, as characterized by the root locus properties of general real rational functions. Our results provide a complete classification of real polynomials of even degree with respect to this conjecture.
\end{abstract}

\maketitle

\section*{Introduction}
The assertion that if a real polynomial $p(x)$ has all real and simple zeros, then $p(x)$ is (locally) strictly monotone was known to Gauss\cite{Gauss7,Gauss8,Sheil}. It can be reformulated as the Laguerre inequality: Let $p(x)$ be a real polynomial with only real zeros. Then $p(x)p^{''}(x)-[p^{'}(x)]^{2}\le0$, $x\in \mathbb{R}$. The Laguerre-P\'{o}lya class (LP) consisting of entire functions obtained as uniform limits on compact sets of sequences of polynomials with only real zeros, also satisfies the Laguerre inequality as well, see \cite{Obreschkoff}.

\medskip
The Wiman and P\'{o}lya conjectures serve to refine the Laguerre inequality. Wiman\cite{Wi1,Wi2} conjectured that if $f$ is a real entire function and both $f$ and $f^{''}$ have only real zeros then $f\in LP$. In \cite{Bergweiler1} W.~Bergweiler, A.~Eremenko, and J.~Langley proved the Wiman's conjecture.

\medskip
P\'{o}lya\cite{P2,P4} formulated his guess precisely by proposing the following two conjectures. The first conjecture is: If the order of the real entire function $f$ is less than 2, and $f$ has only a finite number of non-real zeros, then its derivatives, from a certain one onwards, will have no non-real zeros at all. In \cite{Csordas1}, T.~Craven, G.~Csordas, and W.~Smith settled this conjecture of P\'{o}lya. The second conjecture is: If the order of the real entire function $f$ is greater than 2, and $f$ has only a finite number of non-real zeros, then the number of non-real zeros of $f^{(n)}$ tends to infinity, as $f^{(n)}\to \infty$. In \cite{Bergweiler2}, W.~Bergweiler and A.~Eremenko settled this conjecture.

\medskip
Building on these results, recent work has established conditions under which a real entire function $f$ must belong to the LP-class. These conditions often involve differential polynomials of the form $f(s)f^{''}(s)-\varkappa (f^{'}(s))^{2}$, where $\varkappa$ is a positive real number, see \cite{Nicks}.
The zeros of $f(s)f^{''}(s)-\varkappa (f^{'}(s))^{2}$ for a real function $f$ have been studied in \cite{Langley1,Langley2}. Some relevant studies can be found in \cite{Tao,Ge,Wei,Bhat,Bhowmik1,Bishop,Bao}.

In 2004, Borcea and Shapiro proposed a conjecture in their paper \cite{Shapiro1}. This conjecture, along with their results, emerged from their efforts to prove the "Hawaii Conjecture" by Craven, Csordas, and Smith \cite{Csordas1}, which states the following: Let $p$ be a real polynomial of degree $n \ge 2$. Then, the number of real zeros of $(p^{'}/p)^{'}$ does not exceed the number of non-real zeros of $p$. The Hawaii Conjecture originates from the work of Gauss and Fourier. In their paper \cite{Stephanie1}, titled "Level Sets, a Gauss-Fourier Conjecture, and a Counter-Example to a Conjecture of Borcea and Shapiro," Edwards and Hinkkanen provided an in-depth description of the work of Gauss and Fourier. The contributions of Gauss and Fourier have garnered considerable attention from the mathematical community.

In the correspondence, "G\"{o}ttingische gelehrte Anzeigen" (G.G.A.) from February 25, 1833\cite{Gauss7,Gauss8}, Gauss described some observations and a conjecture of Fourier. Gauss explains that Fourier believed that for each real zero of $(p^{'}/p)^{'}$ (or, as Fourier calls them, "critical points"), there existed an associated pair of non-real zeros of the polynomial $p$.

The Hawaii Conjecture is an attempt by Craven, Csordas, and Smith to quantify the ideas that Gauss and Fourier were investigating. They do not mention an association between zeros. The Hawaii Conjecture was proven in 2011 by M.~Tyaglov, see \cite{Tyaglov1}. Later, B.~Shapiro suggested three new conjectures related to the Hawaii Conjecture; see Conjectures 11, 12, and 13 in \cite{Shapiro0}. In January 2024, our paper\cite{Ma} presented results for Shapiro's Conjecture 12.

Drawing from the historical data compiled by various mathematicians, it is evident that a branch of mathematics has evolved over the past two centuries, rooted in the conjecture of Gauss and Fourier. Shapiro's Conjecture 12 claims the following.

\begin{conjecture} \label{conj:12}
For any real polynomial $p(x)$ of even degree, $$ \sharp_r [(n-1)(p^{'}(x))^{2}-np(x)p^{''}(x)]+\sharp_r p(x)>0.$$
Here, $n$ denotes the degree of $p(x)$ and $\sharp_r p(x)$ represents the number of real zeros of $p(x)$.
\end{conjecture}

In our published work, we employed methods from mathematical analysis to establish that Shapiro's Conjecture 12 holds under three specific conditions. However, under an additional condition, the conjecture may either hold or fail. We begin by reviewing the published results.

$\mathbf{THEOREM}$ $\mathbf{1}$. Let $p(x)$ be a real polynomial of even degree $n$. Then the quantity $ \sharp_r [(n-1)(p^{'}(x))^{2}-np(x)p^{''}(x)]+\sharp_r p(x)>0.$ if and only if one of the following
four cases holds:

(1) the polynomial $p(x)$ has real zeros;

(2) the polynomial $p(x)$ has no real zeros and the polynomial $p^{'}(x)$ has at least three
distinct real zeros;

(3) the polynomial $p(x)$ has no real zeros and the polynomial $p^{'}(x)$ has one real zero
with exponent larger than 1;

(4) the polynomial $p(x)$ has no real zeros, the polynomial $p^{'}(x)$ has one real
zero which is simple, that is, $p^{'}(x)=C(x)(x-w)$, where $C(x)$ is a polynomial
with $C(w)\neq 0$, and the polynomial $(n-1)(C(x))^2(x-w)^2-np(x)C^{'}(x)(x-w)-nC(x)p(x)$ has at least one real zero.

The only case in which the conjecture is false is described in our second result.

$\mathbf{THEOREM}$ $\mathbf{2}$. Let $p(x)$ be a real polynomial of even degree $n$. Then the quantity $ \sharp_r [(n-1)(p^{'}(x))^{2}-np(x)p^{''}(x)]+\sharp_r p(x)=0.$
 if and only if the polynomial $p(x)$
has no real zeros, the polynomial $p^{'}(x)$ has one real zero which is simple, that is, $p^{'}(x)=C(x)(x-w)$, where $C(x)$ is a polynomial with $C(w)\neq 0$, and the polynomial $(n-1)(C(x))^2(x-w)^2-np(x)C^{'}(x)(x-w)-nC(x)p(x)$ has no real zeros.

When the polynomial $(n-1)(C(x))^2(x-w)^2-np(x)C^{'}(x)(x-w)-nC(x)p(x)$ has real zeros, it follows that $ \sharp_r [(n-1)(p^{'}(x))^{2}-np(x)p^{''}(x)]+\sharp_r p(x)>0$. Conversely, when this polynomial $(n-1)(C(x))^2(x-w)^2-np(x)C^{'}(x)(x-w)-nC(x)p(x)$ has no real zeros, we obtain $\sharp_r [(n-1)(p^{'}(x))^{2}-np(x)p^{''}(x)]+\sharp_r p(x)=0$.

Nevertheless, these results do not conclusively resolve Shapiro’s Conjecture 12. Our published work represents only a partial solution to the conjecture. To fully resolve it, we must derive precise and exhaustive conditions governing the presence or absence of real zeros in the polynomial
$(n-1)(C(x))^2(x-w)^2-np(x)C^{'}(x)(x-w)-nC(x)p(x)$, as well as establish two subsequent results. Only with such conditions can the conjecture be considered completely settled.

Deriving the precise and exhaustive conditions governing the presence or absence of real zeros in the polynomial $(n-1)(C(x))^2(x-w)^2-np(x)C^{'}(x)(x-w)-nC(x)p(x)$, deriving these conditions proved unattainable using existing tools in pure mathematics. This impasse led us to adopt a novel mathematical tool: the root locus method from control theory. Our investigation using this method yielded several new insights, most notably Theorem 2.23 in this paper. This theorem revealed the underlying mechanism responsible for the emergence of real critical points in the rational function. Armed with these findings, we were ultimately able to provide a comprehensive solution to Shapiro's Conjecture 12.

We prove Theorem 1.11, which shows that Shapiro's Conjecture 12 holds under nine distinct conditions. Additionally, we prove Theorem 1.12, which demonstrates that the conjecture does not hold under four distinct conditions. Together, these two theorems provide a comprehensive resolution of Shapiro's Conjecture 12.

\section{Definitions and Results}
Let $RF(s)=\frac{\prod_{i=1}^{m} (s-z_{l})^{\gamma_l}}{\prod_{j=1}^{n} (s-p_{j})^{\beta_j}}$ denote any real rational function, where $\gamma_l$ and $\beta_j$ are positive integers.

In textbooks on automatic control theory\cite{Y1,H1,Franklin1,Dorf1}, the root locus equation for a rational function with real constant coefficients is given as follows:

\begin{equation}
K\frac{\prod_{i=1}^{m} (s-z_{l})^{\gamma_l}}{\prod_{j=1}^{n} (s-p_{j})^{\beta_j}}=\pm 1.
\end{equation}

The gain $K$ is defined as: $$K=\abs{\frac{\prod_{j=1}^{n} (s-p_{j})^{\beta_j}}{\prod_{i=1}^{m} (s-z_{l})^{\gamma_l}}}.$$

In this paper, we explore the properties of the root locus in automatic control theory \cite{Y1,H1,Franklin1,Dorf1}. In the complex plane, the points where the argument of a real rational function $RF(s)$ is the same $2q\pi$ degree form curves known as root loci of $2q\pi$ degree. The argument of $RF$ at each point on a specific curve has the same phase angle (argument). The argument of $RF$ at a point on a certain root locus is $2q\pi$ degree. The root loci of $2q\pi$ degree may intersect. Similarly, in the complex plane, the points where the argument of a real rational function $RF(s)$ is the same $2q\pi+\pi$ degree form curves known as root loci of $2q\pi+\pi$ degree. The root loci of $2q\pi+\pi$ degree may intersect. In the root locus method, these intersection points are called breakaway points, which are proven to be the critical points of $RF(s)$.

\begin{definition}
\emph{The breakaway points} of the root locus equation (1.1) are points where at least two root loci intersect.
\end{definition}

\begin{definition}
The root loci of (1.1) intersect on the real axis, where two root loci lie on each side of the real breakaway point on the real axis. Such real breakaway points are called \emph{standard real breakaway points}.
\end{definition}

\begin{definition}
The root loci of (1.1) intersect on the real axis, where only one root locus lies on each side of the real breakaway point on the real axis. Such real breakaway points are called \emph{non-standard real breakaway points}.
\end{definition}

On two sides of a standard real breakaway point, there exist two distinct root loci on the real axis. In contrast, on two sides of a non-standard real breakaway point, there exists only one root locus on the real axis.

\begin{example}
$K\frac{1}{s^4-1}=\pm 1.$
\end{example}
$s^4-1$ has four zeros: $s=\pm 1$, $s=\pm i$. The four root loci of $2q\pi+\pi$ degree emitted by the four zeros of $s^4-1$ intersect at $s=0$. There are two root loci on the real axis. One is emitted from the real zero $s=1$, and the other is emitted from the real zero $s=-1$. The other two root loci are emitted from the non-real zeros $s=\pm i$. The point $s=0$ is a breakaway point.

Because the two root loci emitted from the real zeros $s=\pm 1$ intersect at point $s=0$ on the real axis. The standard breakaway point $s=0$ is generated.

$RF_1(s)=\frac{1}{s^4-1}$, $RF_1^{'}(s)=\frac{-s^3}{(s^4-1)^2}$, so $s=0$ is a critical point has multiplicity three.

\begin{example}
$K\frac{1}{s^3-1}=\pm 1.$
\end{example}
$s^3-1$ has three zeros: $s=1$, $s=\frac{-1\pm i\sqrt{3}}{2}$. The three root loci of $2q\pi+\pi$ degree emitted by the three zeros of $s^3-1$ intersect at $s=0$. There is only one root locus on the real axis. It is emitted from the real zero $s=1$. The other two root loci are emitted from the non-real zeros $s=\frac{-1\pm i\sqrt{3}}{2}$. On two sides of $s=0$, there exists only one root locus on the real axis. The non-standard breakaway point $s=0$ is generated.

$RF_2(s)=\frac{1}{s^3-1}$, $RF_2^{'}(s)=\frac{-s^2}{(s^3-1)^2}$, so $s=0$ is a critical point has multiplicity two.

Let $p$ be a polynomial defined as: $p(s)=s^n+a_1s^{n-1}+\cdots+a_n$. Its first and second derivatives are: $p^{'}(s)=ns^{n-1}+a_1(n-1)s^{n-2}+\cdots+a_{n-1}$. $p^{''}(s)=n(n-1)s^{n-2}+a_1(n-1)(n-2)s^{n-3}+\cdots+a_{n-2}$. Define the polynomial: $$\Delta=(n-1)(p^{'}(x))^{2}-np(x)p^{''}(x)$$

To find the roots of $\Delta$, we express $\Delta$ in terms of a rational function: $\frac{np^{''}p}{(n-1)(p^{'})^2}$, whose constant coefficient is $K_0=\frac{n}{n-1}$. Using the rational function $\frac{p^{''}p}{(p^{'})^2}$, we establish the root locus equation:
\begin{equation}
K\frac{p^{''}p}{(p^{'})^2}=\pm 1.
\end{equation}
On the root locus of $2q\pi$ degree of the root locus equation (1.2), the point with gain $K=K_0$ is a root of  $\Delta$. On the root locus of $2q\pi+\pi$ degree of (1.2), the point with gain $K=K_0$ is not a root of $\Delta$, it is not a root.

Let $PP$ denote the rational function $\frac{p^{''}p}{(p^{'})^2}$, where $p$ is any real polynomial of even degree $n$. Define $\Lambda$ as the set of all such $PP$.

We partition $\Lambda$ into two subsets based on whether $p(x)$ has real zeros:

$\Lambda_1$: $\Lambda_1 \subseteq \Lambda$, in which $p$ has real zeros.

$\Lambda_2$: $\Lambda_2 \subseteq \Lambda$, in which $p$ has no real zeros.

These subsets satisfy: $\Lambda=\Lambda_1\cup\Lambda_2$ and $\Lambda_1\cap\Lambda_2=\phi$.

We subdivide $\Lambda_2$ into three subcases based on the number of distinct real zeros of $p^{'}(x)$:

$\Lambda_{21}$: $\Lambda_{21} \subseteq \Lambda_2$, in which $p^{'}(x)$ has at least three real zeros. If $p^{'}(x)$ has real multiple zeros, it has at least two distinct real zeros.

$\Lambda_{22}$: $\Lambda_{22} \subseteq \Lambda_2$, in which $p^{'}(x)$ has one real zero
with exponent larger than 1.

$\Lambda_{23}$: $\Lambda_{23} \subseteq \Lambda_2$, in which $p^{'}(x)$ has one real simple
zero.

These subsets satisfy: $\Lambda_2=\Lambda_{21}\cup\Lambda_{22}\cup\Lambda_{23}$, $\Lambda_{21}\cap\Lambda_{22}=\phi$, $\Lambda_{21}\cap\Lambda_{23}=\phi$ and $\Lambda_{22}\cap\Lambda_{23}=\phi$.

Let $\Gamma=\Lambda_{23}$. Within $\Gamma$, we derive the conditions under which the polynomial $(n-1)(C(x))^2(x-w)^2-np(x)C^{'}(x)(x-w)-nC(x)p(x)$
either has real zeros or has no real zeros.

In $\Gamma$,  $p^{'}$ has exactly one real simple zero, denoted $p_0$.

We partition $\Gamma$ into two subsets based on whether $p^{''}$ has real zeros:

$\Gamma_1$: $\Gamma_1 \subseteq \Gamma$, in which $p^{''}$ has no real zeros.

$\Gamma_2$: $\Gamma_2 \subseteq \Gamma$, in which $p^{''}$ has a real zero.

These subsets satisfy: $\Gamma=\Gamma_1\cup\Gamma_2$ and $\Gamma_1\cap\Gamma_2=\phi$.

In $\Gamma_1$, $p^{''}$ has no real zeros.  The root loci of $PP$ may or may not have a standard real breakaway point. Accordingly, we further divide $\Gamma_1$ into two subsets:

$\Gamma_{11}$: $\Gamma_{11} \subseteq \Gamma_1$, in which the root loci of $PP$ have no standard real breakaway points.

$\Gamma_{12}$: $\Gamma_{12} \subseteq \Gamma_1$, in which the root loci of $PP$ have at least one standard real breakaway point.

These subsets satisfy: $\Gamma_1=\Gamma_{11}\cup\Gamma_{12}$ and $\Gamma_{11}\cap\Gamma_{12}=\phi$.

In $\Gamma_2$, $p^{''}$ has real zeros. Based on the distribution of real zeros for $p^{''}$ on the left and right sides of $p_0$, we partition the set $\Gamma_2$ into three subsets.

$\Gamma_{21}$: $\Gamma_{21} \subseteq \Gamma_2$. On the real axis, to the right of $p_0$, $p^{''}$ has real zeros, while to the left of $p_0$, $p^{''}$ has no real zeros.

$\Gamma_{22}$: $\Gamma_{22} \subseteq \Gamma_2$. On the real axis, to the right of $p_0$, $p^{''}$ has no real zeros, while to the left of $p_0$, $p^{''}$ has real zeros.

$\Gamma_{23}$: $\Gamma_{23} \subseteq \Gamma_2$. On the real axis to the right of $p_0$, $p^{''}$ has real zeros. On the real axis to the left of $p_0$, $p^{''}$ has real zeros.

These subsets satisfy: $\Gamma_ {2}=\Gamma_{21}\cup\Gamma_{22}\cup\Gamma_{23}$. $\Gamma_{21}\cap\Gamma_{22}=\phi$. $\Gamma_{21}\cap\Gamma_{23}=\phi$.  $\Gamma_{22}\cap\Gamma_{23}=\phi$.

When $p^{''}$ has no real zeros, the solution for the Shapiro Conjecture 12 can be obtained using the sets $\Gamma_{11}$, $\Gamma_{12}$, and the subset of $\Gamma_{12}$ defined below. However, when $p^{''}$ has real zeros, the Shapiro Conjecture 12 becomes significantly more complex. Our research requires more detailed analysis. A key problem we must investigate is whether real zeros of $\Delta$ exist in the intervals between adjacent real zeros of $p^{''}$. Therefore, we define below the three types of intervals that may exist between adjacent real zeros of $p^{''}$. Simultaneously, we must also examine whether real zeros of $\Delta$ exist in these three types of intervals.

On the real axis to the right of $p_0$, $p^{''}$ has real zeros. Let $z_m$ be the largest real zero of $p^{''}$ to the right of $p_0$. The infinite interval from $z_m$ to positive infinity is $(z_m,+\infty)$. According to results in Section 2, we can obtain: $(z_m,+\infty)$ must be a root locus of $2q\pi$ degree of (1.2).

\begin{definition}
The interval $(z_m,+\infty)$ is called \emph{the right infinite interval}.
\end{definition}

On the real axis to the left of $p_0$, $p^{''}$ has real zeros. Let $z_s$ be the smallest real zero of $p^{''}$ to the left of $p_0$. The infinite interval from $z_s$ to negative infinity is $(-\infty, z_s)$.

\begin{definition}
If $(-\infty, z_s)$ is a root locus of $2q\pi$ degree of (1.2), it is called \emph{the left infinite interval of $2q\pi$ degree}.
\end{definition}

For any two adjacent real zeros $z_1$ and $z_2$ of $p^{''}$ either both to the right or both to the left of $p_0$. Such intervals cannot contain $p_0$ and must lie entirely on one side of $p_0$.

\begin{definition}
If $(z_1, z_2)$  is the root locus of $2q\pi$ degree of (1.2), it is called \emph{the finite interval of $2q\pi$ degree}.
\end{definition}
In this paper, we prove the following: On the root loci of $2q\pi$ degree of (1.1), the standard real breakaway points of the root loci of the $2q\pi$ degree of (1.1) are the extreme points of the gain of (1.1). This leads to the following two definitions:

\begin{definition}
If $K$ attains a maximum at a standard real breakaway point of (1.1), it is called \emph{a maximum real breakaway point}.
\end{definition}

\begin{definition}
If $K$ attains a minimum at a standard real breakaway point of (1.1), it is called \emph{a minimum real breakaway point}.
\end{definition}

(1.1) includes (1.2), i.e. (1.2) is a special case of (1.1). Thus, the results satisfied by (1.1) must be satisfied by (1.2).

When (1.2) has standard real breakaway points, we partition the set $\Gamma_{12}$ into two subsets based on whether the gains at all maximum real breakaway points are less than $K_0$:

$\Gamma_{121}$: $\Gamma_{121} \subseteq \Gamma_{12}$, in which the gains at all maximum real breakaway points are less than $K_0$.

$\Gamma_{122}$: $\Gamma_{122} \subseteq \Gamma_{12}$, in which there exists at least one maximum real breakaway point $b_{122}$ such that $K(b_{122})\ge K_0$.

These subsets satisfy: $\Gamma_ {12}=\Gamma_{121}\cup\Gamma_{122}$, $\Gamma_{121}\cap\Gamma_{122}=\phi$.

We further partition $\Gamma_{21}$ into two subsets based on whether the number of real zeros of $p^{''}$ to the right of $p_0$ is odd or even.

$\Gamma_{211}$: $\Gamma_{211} \subseteq \Gamma_{21}$, in which $p^{''}$ has an even number of real zeros to the right of $p_0$ on the real axis.

$\Gamma_{212}$: $\Gamma_{212} \subseteq \Gamma_{21}$, in which $p^{''}$ has an odd number of real zeros to the right of $p_0$ on the real axis.

These subsets satisfy: $\Gamma_ {21}=\Gamma_{211}\cup\Gamma_{212}$, $\Gamma_{211}\cap\Gamma_{212}=\phi$.

1. In the right infinite interval $(z_m,+\infty)$, there is either no standard real breakaway point, or the gains of all minimum real breakaway points are greater than $K_0$. If either of these two cases holds, the requirement that $PP$ must satisfy  in $(z_m,+\infty)$ is fulfilled.

2. On the real axis to the right of $p_0$, if finite intervals of $2q\pi$ degree exist, then in all such finite intervals, the gains of all minimum real breakaway points must also be greater than $K_0$.

The rational functions $PP$ in $\Gamma_{212}$ that satisfy all two conditions form a distinct subclass, constituting another subset.

Conversely, if in at least one interval among the right infinite interval and all finite intervals of $2q\pi$ degrees (if they exist), there exists at least one minimum real breakaway point $b_{2122}$ such that $K(b_{2122})\le K_0$. The rational functions $PP$ satisfying this condition form another class. This class of rational functions $PP$ constitutes a distinct subset.

Based on the preceding analysis, we partition $\Gamma_{212}$ into two subsets:

$\Gamma_{2121}$: $\Gamma_{2121} \subseteq \Gamma_{212}$. The subset satisfying: In $(z_m,+\infty)$, there is no standard real breakaway point, or all minimum real breakaway points have gains greater than $K_0$. Additionally, on the real axis to the right of $p_0$, if finite intervals of $2q\pi$ degree exist, all minimum real breakaway points in these intervals must also have gains greater than $K_0$.

$\Gamma_{2122}$: $\Gamma_{2122} \subseteq \Gamma_{212}$. On the real axis to the right of $p_0$, in the right infinite interval or at least one interval among all finite intervals of $2q\pi$ degrees (if they exist), there exists at least one minimum real breakaway point $b_{2122}$ such that $K(b_{2122})\le K_0$.

These subsets satisfy: $\Gamma_ {212}=\Gamma_{2121}\cup\Gamma_{2122}$, $\Gamma_{2121}\cap\Gamma_{2122}=\phi$.

The set $\Gamma_{23}$ is partitioned into two subsets based on whether the parity of the number of real zeros of $p^{''}$ to the right of $p_0$ is odd or even.

$\Gamma_{231}$: $\Gamma_{231} \subseteq \Gamma_{23}$, in which $p^{''}$ has an even number of real zeros to the right of $p_0$ on the real axis.

$\Gamma_{232}$: $\Gamma_{232} \subseteq \Gamma_{23}$, in which $p^{''}$ has an odd number of real zeros to the right of $p_0$ on the real axis.

These subsets satisfy $\Gamma_ {23}=\Gamma_{231}\cup\Gamma_{232}$ and $\Gamma_{231}\cap\Gamma_{232}=\phi$.

1. In the right-infinite interval $(z_m,+\infty)$, either there is no standard real breakaway point, or all minimum real breakaway points have  gains greater than $K_0$. If either of these two cases holds, the requirement that $PP$ must satisfy  in $(z_m,+\infty)$ is fulfilled.

2. On the real axis to the right of $p_0$, if $2q\pi$-degree finite intervals exist, then in all such intervals, the gains of all minimum real breakaway points must be greater than $K_0$.

3. On the real axis to the left of $p_0$, if $2q\pi$-degree finite intervals exist, all minimum real breakaway points in these intervals must also have gains greater than $K_0$.

4. If the left-infinite interval of $2q\pi$-degree $(-\infty, z_s)$ exists, either there is no standard real breakaway point, or all minimum real breakaway points must have gains greater than $K_0$. If either of these two cases holds, the requirement that $PP$ must satisfy  in the left-infinite interval of $2q\pi$-degree $(-\infty, z_s)$ is fulfilled.

The rational functions $PP$ in $\Gamma_{232}$ that satisfy all four conditions form a distinct subclass, constituting another subset.

1. In the right-infinite interval $(z_m,+\infty)$, there exists at least one minimum real breakaway point $b_{m2322}$ such that $K(b_{m2322})\le K_0$.

2. To the right of $p_0$, if $2q\pi$-degree finite intervals exist, in at least one such finite interval, there exists at least one minimum real breakaway point $b_{r2322}$ such that $K(b_{r2322})\le K_0$.

3. If a left-infinite interval of $2q\pi$-degree $(-\infty, z_s)$ exists, it contains at least one minimum real breakaway point $b_{s2322}$ such that $K(b_{s2322})\le K_0$.

4. To the left of $p_0$, if $2q\pi$-degree finite intervals exist, in at least one such finite interval, there exists at least one minimum real breakaway point $b_{l2322}$ such that $K(b_{l2322})\le K_0$.

The rational functions $PP$ in $\Gamma_{232}$ that satisfy any one of these four conditions form a distinct subclass, constituting another subset.

Based on the preceding analysis, we partition $\Gamma_{232}$ into two subsets:

$\Gamma_{2321}$: $\Gamma_{2321} \subseteq \Gamma_{232}$. In $(z_m,+\infty)$, either there is no standard real breakaway point, or all minimum real breakaway points have a gain greater than $K_0$. On the real axis to the right of $p_0$, if finite intervals of $2q\pi$ degree exist, in all such intervals, all minimum real breakaway points have a gain greater than $K_0$. On the real axis to the left of $p_0$, if finite intervals of $2q\pi$ degree exist, in all such finite intervals, all minimum real breakaway points have a gain greater than $K_0$. Moreover, if a left-infinite interval $(-\infty, z_s)$ of $2q\pi$ degree exists, in that left infinite interval of $2q\pi$, either there is no standard real breakaway point, or all minimum real breakaway points have a gain greater than $K_0$.

$\Gamma_{2322}$: $\Gamma_{2232} \subseteq \Gamma_{232}$. On the real axis to the right of $p_0$, if there exists a finite interval of $2q\pi$ degree, in the right-infinite interval or in at least one of the finite intervals of $2q\pi$ degrees, there exists at least one minimum real breakaway point, $b_{m2322}$ with $K(b_{m2322})\le K_0$; or, $b_{r2322}$ with $K(b_{r2322})\le K_0$. Or, on the real axis to the left of $p_0$, if there exists the left-infinite interval or a finite interval of $2q\pi$ degree, in at least one such interval, at least one minimum real breakaway point exists, $b_{s2322}$ with $K(b_{s2322})\le K_0$; or, $b_{l2322}$ with $K(b_{l2322})\le K_0$.

These subsets satisfy $\Gamma_ {232}=\Gamma_{2321}\cup\Gamma_{2322}$, $\Gamma_{2321}\cap\Gamma_{2322}=\phi$.

To sum up, we partition $\Gamma$ as follows:

\begin{itemize}
\item $\Lambda_1$
\item $\Lambda_2$
  \begin{itemize}
  \item $\Lambda_{21}$
  \item $\Lambda_{22}$
  \item $\Lambda_{23}$
  \end{itemize}
\end{itemize}

Let $\Gamma=\Lambda_{23}$. We further partition $\Gamma$ based on properties of $p''$:

\begin{itemize}
\item $\Gamma_1$
  \begin{itemize}
  \item $\Gamma_{11}$ (Lemma 3.4)
  \item $\Gamma_{12}$
    \begin{itemize}
    \item $\Gamma_{121}$ (Lemma 3.6)
    \item $\Gamma_{122}$ (Lemma 3.7)
    \end{itemize}
  \end{itemize}
\item $\Gamma_2$
  \begin{itemize}
  \item $\Gamma_{21}$
    \begin{itemize}
    \item $\Gamma_{211}$ (Lemma 3.22)
    \item $\Gamma_{212}$
      \begin{itemize}
      \item $\Gamma_{2121}$ (Lemma 3.20)
      \item $\Gamma_{2122}$ (Lemma 3.21)
      \end{itemize}
    \end{itemize}
  \item $\Gamma_{22}$ (Lemma 3.8)
  \item $\Gamma_{23}$
    \begin{itemize}
    \item $\Gamma_{231}$ (Lemma 3.23)
    \item $\Gamma_{232}$
      \begin{itemize}
      \item $\Gamma_{2321}$ (Lemma 3.25)
      \item $\Gamma_{2322}$ (Lemma 3.26)
      \end{itemize}
    \end{itemize}
  \end{itemize}
\end{itemize}

\begin{theorem}
Let $p(x)$ be a real polynomial of even degree $n$. Then, the quantity
$$\sharp_r [(n-1)(p^{'}(x))^{2}-np(x)p^{''}(x)]+\sharp_r p(x)>0$$
if and only if
$$ PP \in \Lambda_1 \cup\ \Lambda_{21}\cup\Lambda_{22}\cup\Gamma_{122}\cup\Gamma_{22}\cup\Gamma_{2122}\cup\Gamma_{211}\cup\Gamma_{231}\cup\Gamma_{2322}.$$
\end{theorem}

\begin{theorem}
Let $p(x)$ be a real polynomial of even degree $n$. Then, the quantity
$$\sharp_r [(n-1)(p^{'}(x))^{2}-np(x)p^{''}(x)]+\sharp_r p(x)=0$$
if and only if
$$PP \in \Gamma_{11}\cup\Gamma_{121}\cup\Gamma_{2121}\cup\Gamma_{2321}.$$
\end{theorem}
These cases cover all possibilities for even-degree real polynomials.

\section{Proof of Necessary and Sufficient Condition for Real Critical Points}

When the right side of (1.1) equals $1$, the root locus equation (1.1) and its corresponding root loci are both of $2q\pi$ degrees. Conversely, when the right side of (1.1) equals $-1$, the root locus equation (1.1) and its corresponding root loci are both of $2q\pi+\pi$ degrees, where $q=0, \pm 1,\pm 2,\cdots$ is an integer. At points where the argument of $RF(s)$ is $2q\pi$ or $2q\pi+\pi$, $RF(s)$ takes real values.

The results of the root loci of (1.1) are well-established in the textbooks on automatic control theory. While some textbooks provide rigorous and mathematically sound proofs (e.g., \cite{Y1,H1}), others do not meet the same standard (e.g., \cite{Franklin1,Dorf1}). Nevertheless, the root locus method has been widely used for over 70 years and has been included in textbooks for decades, making it a reliable and trusted tool. For this reason, we do not reproduce their proofs here.

When the right side of (1.1) equals $1$, in the complex plane, the roots of (1.1)  form curves known as root loci of $2q\pi$ degree. The root loci of $2q\pi$ degree may intersect. When the right side of (1.1) equals $-1$, in the complex plane, the roots of (1.1)  form curves known as root loci of $2q\pi+\pi$ degree. The root loci of $2q\pi+\pi$ degree may intersect. However, a root locus of $2q\pi+\pi$ degree does not intersect a root locus of $2q\pi$ degree.

$\mathbb{C}\cup\{\infty\}$ denotes the extended complex plane. Let $\Phi\subset\mathbb{C}\cup\{\infty\}$ denote the set of points $s=\sigma+it\in\mathbb{C}\cup\{\infty\}$ that are neither zeros nor poles of (1.1). $s=\sigma+it\in\Phi$ is an arbitrary point in $\Phi$.
$\mathbb{C}^{*}=\{s=x+iy: y\neq 0\}$ denotes the extended complex plane excluding the real axis. $\sum^n_{j=1} \beta_j$ denotes the total number of poles of (1.1). $\sum^m_{l=1} \gamma_l$ denotes the total number of zeros of (1.1). (counting multiplicity)

The root locus in $\mathbb{C}\cup\{\infty\}$ has several key properties. These properties govern its distribution patterns, which in turn provide insight into both the shape of the root locus and other results related to $RF(s)$.
\begin{lemma}
The $2q\pi+\pi$  and $2q\pi$ degree root loci of (1.1) begin at the poles of (1.1) or at infinity and end at the zeros of (1.1) or at infinity.

1. When $\sum^n_{j=1} \beta_j > \sum^m_{l=1} \gamma_l$, there are $\sum^n_{j=1} \beta_j- \sum^m_{l=1} \gamma_l$ branches of the $2q\pi+\pi$  and $2q\pi$ degree root loci ending at the infinity in $\mathbb{C}\cup\{\infty\}$.

2. When $\sum^n_{j=1} \beta_j < \sum^m_{l=1} \gamma_l$, there are $ \sum^m_{l=1} \gamma_l - \sum^n_{j=1} \beta_j$ branches of the $2q\pi+\pi$  and $2q\pi$ degree root loci beginning at the infinity in $\mathbb{C}\cup\{\infty\}$.
\end{lemma}

\begin{lemma}
A branch of the root loci is an entire root locus from the starting point that extends at the ending point(including infinity).

1. When $\sum^n_{j=1} \beta_j \ge \sum^m_{l=1} \gamma_l$, there are $\sum^n_{j=1} \beta_j$ branches of the $2q\pi+\pi$  and $2q\pi$ degree root loci. The $\sum^n_{j=1} \beta_j$ branches of the $2q\pi+\pi$  and $2q\pi$ degree root loci are symmetrical with respect to the real axis.

2. When $\sum^n_{j=1} \beta_j < \sum^m_{l=1} \gamma_l$, there are $\sum^m_{l=1} \gamma_l$ branches of the $2q\pi+\pi$  and $2q\pi$ degree root loci. The $\sum^m_{l=1} \gamma_l$ branches of the $2q\pi+\pi$  and $2q\pi$ degree root loci are symmetrical with respect to the real axis.
\end{lemma}

\begin{lemma}
The real axis is the $2q\pi+\pi$ and $2q\pi$ degree root loci of (1.1).

1. The necessary and sufficient condition that an interval on the real axis must be a $2q\pi+\pi$ degree root locus is that the total number of real poles and zeros of (1.1) on the right side of this interval is an odd number.

2. The necessary and sufficient condition that an interval on the real axis must be a $2q\pi$ degree root locus is that the total number of real poles and zeros of (1.1) on the right side of this interval is an even number.
\end{lemma}

\begin{lemma}
The breakaway points of the root loci of (1.1) satisfy (2.1).

\begin{equation}
\frac{dK(s)}{ds}\mid_{s=s_0}=0.
\end{equation}
\end{lemma}

Lemmas 1.1, 1.2, 1.3 and 1.4 are adapted from textbooks on automatic control theory \cite{Y1,H1,Franklin1,Dorf1}. The remaining results are original discoveries and proofs by the authors.

\begin{lemma}
After removing coincident poles and zeros in (1.1):

1. All finite poles $p_j$ of $RF$ are also finite poles of (1.1) with $K=0$.

2. All finite zeros $z_l$ of $RF$ are also finite zeros of (1.1) with $K=+\infty$.
\end{lemma}
\begin{proof}
1. Rewriting (1.1) as:
$$K\prod^{m}_{l=1} (s-z_l)^{\gamma_l}=(a+ib)\prod^{n}_{j=1}  (s-p_j)^{\beta_j}.$$ All finite poles of $RF$ correspond to zeros of $\prod^{n}_{j=1}  (s-p_j)^{\beta_j}$. If $K=0$, all finite poles of $RF$ are roots of this equation. Conversely, all finite roots of this equation when $K=0$ are finite poles of $RF$. Thus, at finite poles of (1.1), at finite poles of the $RF$, $K=0$.

2. Similarly, equation (1.1) can be expressed as
$$\prod^{m}_{l=1} (s-z_l)^{\gamma_l}=\frac{(a+ib)}{K}\prod^{n}_{j=1}  (s-p_j)^{\beta_j}.$$
If $K=+\infty$, all finite zeros of $RF$ are roots of the equation. Conversely, all roots of the equation when $K=+\infty$ are finite zeros of $RF$. Therefore, at finite zeros of (1.1), at finite zeros of the $RF$, $K=+\infty$.
\end{proof}

The continuity of $K$ with respect to $s\in\mathbb{C}\cup\{\infty\}$ follows directly from the expression of $K$. Applying Lemma 2.1 and Lemma 2.5, Theorem 2.6 is immediate.
\begin{theorem}
For any $s\in\mathbb{C}\cup\{\infty\}$, the value $K$ takes all non-negative real numbers, ranging from 0 at poles of (1.1) to $+\infty$ at zeros of (1.1).
\end{theorem}

\begin{lemma}
Assume that any two points $s_1$ and $s_2$ are on an arbitrary root locus of (1.1). Then $K(s_1)\neq K(s_2)$.  $s_1,s_2\in\Phi$.
\end{lemma}
\begin{proof}
When $\sum^n_{j=1} \beta_j \ge \sum^m_{l=1} \gamma_l$. Since (1.1) has $\sum^n_{j=1} \beta_j$ poles, each pole emits a root locus of (1.1) of $2q\pi$ degree. Thus, (1.1) emits $\sum^n_{j=1} \beta_j$ root loci, each of $2q\pi$ degree. By Theorem 2.6, the gain values of points on these root loci with degree $2q\pi$ range from 0 at the poles to positive infinity at the zeros. Therefore, on each root locus of the $\sum^n_{j=1} \beta_j$ root loci, there exists at least one point with gain value $\frac{1}{\abs{FF}}$. $\frac{1}{\abs{FF}}\neq 0, \infty$. Therefore, there are at least $\sum^n_{j=1} \beta_j$ points with phase angle $2q\pi$ and gain value $\frac{1}{\abs{FF}}$.

There are at least $\sum^n_{j=1} \beta_j$ points that satisfy the equation $\frac{1}{\abs{FF}}RF(s)=1$, which implies $RF(s)=\abs{FF}$. Because the gain is a positive real number. So, let $FF=\abs{FF}$, $FF$ is a non-zero finite positive real constant. Therefore, $RF(s)=FF$. $FF\prod^{n}_{j=1}(s-p_{j})^{\beta_j}-\prod^{m}_{l=1}(s-z_{l})^{\gamma_l}=0$, we obtain a $\sum^n_{j=1} \beta_j$ order polynomial equation. This equation has exactly $\sum^n_{j=1} \beta_j$ roots. Hence, for any two points $s_1$ and $s_2$ on an arbitrary root locus of (1.1), $K(s_1)\neq K(s_2)$.

For the case $\sum^n_{j=1} \beta_j < \sum^m_{l=1} \gamma_l$. Since (1.1) has $\sum^m_{l=1} \gamma_l$ zeros, each zero receives a root locus of (1.1) with degree $2q\pi$. Thus, (1.1) receives $\sum^m_{l=1} \gamma_l$ root loci of (1.1) with degree $2q\pi$. By repeating the above proof, we can obtain the result for $\sum^n_{j=1} \beta_j < \sum^m_{l=1} \gamma_l$.

By repeating the above proof, we can obtain the result for the $2q\pi+\pi$ degree root locus.
\end{proof}

A pole emits a root locus in this paper means that a root locus origins from the pole.

\begin{theorem}
On each root locus of (1.1), when the point $s$ moves from the poles of (1.1) to the zeros of (1.1), then gains are strictly and monotonically increasing.
\end{theorem}
\begin{proof}
If the gains $|K(s)|$ are not monotonic on the root locus of (1.1), there exist two points $s_1$ and $s_2$ on the same root locus such that $|K(s_1)|=|K(s_2)|$. This would contradict Lemma 2.7. By Theorem 2.6, the gain must be strictly monotonic. We obtain Theorem 2.8.
\end{proof}

\begin{theorem}
Two root loci of (1.1) with distinct unit complex values $1$ and $-1$ cannot intersect in $\mathbb{C}\cup\{\infty\}$.
\end{theorem}
\begin{proof}
Assume that two root loci of (1.1) with distinct unit complex values $1$ and $-1$ intersect at a point $s_0$, which is neither a zero nor a pole of (1.1).

Since $s_0$ is the same point, its gain $K_0$ should be the same. $s_0$ allows the following two equations to hold simultaneously. $K_0\frac{\prod^{m}_{l=1} (s_0-z_l)^{\gamma_l}}
{\prod^{n}_{j=1}  (s_0-p_j)^{\beta_j}}=1$, and $K_0\frac{\prod^{m}_{l=1} (s_0-z_l)^{\gamma_l}}{\prod^{n}_{j=1}  (s_0-p_j)^{\beta_j}}=-1$.
Subtracting these equations yields $0=2$. $0=2$ cannot hold. Since the contradiction arises from the assumption that the two root loci intersect at $s_0$, this assumption is false.
\end{proof}

\begin{theorem}
Two root loci of (1.1) with identical unit complex values $1$ or $-1$ may intersect in $\mathbb{C}\cup\{\infty\}$.
\end{theorem}
\begin{proof}
Consider the root locus equation $\frac{K}{(s-1)(s-3)}=\pm 1$. The interval $[1,3]$ contains root loci of $2q\pi+\pi$ degree from both poles at $s=1$ and $s=3$, each with a unit complex value of $-1$. These root loci intersect at point $s=2$.

By applying Rolle's theorem, there exists a critical point of the function $(s-1)(s-3)$ in the interval $[1,3]$. This critical point of $(s-1)(s-3)$ is $s=2$, which is the intersection point of the root loci.
\end{proof}

\begin{lemma}
All points on the real axis are roots of the root locus equation (1.1). Except for zeros and poles of (1.1), each finite point on the real axis lies on a root locus of (1.1). The root loci of (1.1) fill the entire real axis.
\end{lemma}

\begin{proof}
In the proof of Lemma 2.5, we proved that at the zeros of (1.1), $K=+\infty$. The zeros of (1.1) satisfy (1.1) with $K=+\infty$. At the poles of (1.1), $K=0$. The poles of (1.1) satisfy (1.1) with $K=0$.

Any finite point $s_{a}$ on the real axis is neither a zero nor a pole of (1.1). Substituting $s_{a}$ into the gain expression $K=|\frac{\prod^{n}_{j=1}  (s-p_j)^{\beta_j}}{\prod^{m}_{l=1} (s-z_l)^{\gamma_l}}|$ yields a  positive finite real number. Substituting $s_{a}$ into the phase angle expression $\varphi(\sigma,t)$ yields a finite certain phase angle $2q\pi$ or $2q\pi+\pi$. Therefore, the finite point $s_{a}$ on the real axis must lie on a root locus of  $\varphi(\sigma_a,t_a)$ degree of (1.1) with a gain $K(s_{a})$ that is a  positive finite real number. $s_{a}$ satisfies (1.1) with the gain $K(s_{a})$.

The two points at infinity on the real axis can be zeros, poles, or general points with finite gain values $A$. Repeating the previous proof, we conclude that these points at infinity on the real axis lie on a root locus of (1.1) and satisfy (1.1) with a specific gain.
\end{proof}

\begin{definition}
If two or more root loci intersect in $\mathbb{C}\cup\{\infty\}$, these intersecting root loci are called \emph{"common root loci"}.
\end{definition}

Multiplying both sides of (1.1) by $\prod (s-p_j)^{\beta_j}$ gives:
$K(s) \prod_{l=1}^m (s-z_l)^{\gamma_l} =\pm \prod_{j=1}^n (s-p_j)^{\beta_j},$
  which is a polynomial whose roots coincide with solutions of (1.1). All non-zero factors are moved to one side. This transformation yields the characteristic equation of (1.1), which can be written as: $K(s)\prod^{m}_{l=1}(s-z_{l})^{\gamma_{l}}-(\pm)
\prod^{n}_{j=1}(s-p_{j})^{\beta_{j}}=0$. The characteristic equation of (1.1) is equivalent to (1.1). Therefore, the roots of the characteristic equation are the roots of (1.1), and vice versa. We have Lemma 2.13.

\begin{lemma}
Equation (1.1) can be transformed into its characteristic equation. The roots of (1.1) and its characteristic equation are identical.
\end{lemma}

The gain function $K(s)=\pm \frac{\prod_{j=1}^{n} (s-p_{j})^{\beta_j}}{\prod_{i=1}^{m} (s-z_{i})^{\gamma_l}}$ must be non-negative. Thus, the sign $\pm$ in front of the  function $\frac{\prod_{j=1}^{n} (s-p_{j})^{\beta_j}}{\prod_{i=1}^{m} (s-z_{i})^{\gamma_l}}$ is determined by the sign of $\frac{\prod_{j=1}^{n} (s-p_{j})^{\beta_j}}{\prod_{i=1}^{m} (s-z_{i})^{\gamma_l}}$: it is positive when the fraction $\frac{\prod_{j=1}^{n} (s-p_{j})^{\beta_j}}{\prod_{i=1}^{m} (s-z_{i})^{\gamma_l}}$ is positive and negative when the fraction $\frac{\prod_{j=1}^{n} (s-p_{j})^{\beta_j}}{\prod_{i=1}^{m} (s-z_{i})^{\gamma_l}}$ is negative. Let $(\pm)\big|{s=s_0}$ denote the sign of the gain function when a finite real point $s_0$ is substituted into the expression $K(s)$.

\begin{lemma}
For any finite point $s_0$ on the real axis, the left-hand side of the characteristic equation with gain $K(s_0)$ can be rewritten as another expression: $(s-s_0)^{\gamma_0}g(s)$, where $g(s_0)\neq0$, $g(s_0)\neq\infty$ and $g^{'}(s_0)\neq\infty$.
\end{lemma}

\begin{proof}
By Lemma 2.11, the finite point $s_0$ lies on the root loci. Substituting the gain $K(s_0)$ of the finite point $s_0$ into the characteristic equation yields: $(\pm)\mid{_{s=s_0}}\frac{\prod^{n}_{j=1}(s_0-p_{j})^{\beta_{j}}}
{\prod^{m}_{l=1}(s_0-z_{l})^{\gamma_{l}}}\prod^{m}_{l=1}(s-z_{l})^{\gamma_{l}}-
(\pm)\prod^{n}_{j=1}(s-p_{j})^{\beta_{j}}=0$.
Since $s_0$ is a root of the characteristic equation, the left-hand side of the characteristic equation can be expressed as: $(s-s_0)^{\gamma_0}g(s)$. $g(s_0)\neq 0$. $(\pm)\mid{_{s=s_0}}\frac{\prod^{n}_{j=1}(s_0-p_{j})^{\beta_{j}}}
{\prod^{m}_{l=1}(s_0-z_{l})^{\gamma_{l}}}\prod^{m}_{l=1}(s-z_{l})^{\gamma_{l}}-
(\pm)\prod^{n}_{j=1}(s-p_{j})^{\beta_{j}}=(s-s_0)^{\gamma_0}g(s)=0$.

Because the function $(\pm)\mid{_{s=s_0}}\frac{\prod^{n}_{j=1}(s_0-p_{j})^{\beta_{j}}}
{\prod^{m}_{l=1}(s_0-z_{l})^{\gamma_{l}}}\prod^{m}_{l=1}(s-z_{l})^{\gamma_{l}}-
(\pm)\prod^{n}_{j=1}(s-p_{j})^{\beta_{j}}$ on the left side of the characteristic equation is a polynomial. After factoring out $(s-s_0)^{\gamma_0}$ in the polynomial $(\pm)\mid{_{s=s_0}}\frac{\prod^{n}_{j=1}(s_0-p_{j})^{\beta_{j}}}
{\prod^{m}_{l=1}(s_0-z_{l})^{\gamma_{l}}}\prod^{m}_{l=1}(s-z_{l})^{\gamma_{l}}-
(\pm)\prod^{n}_{j=1}(s-p_{j})^{\beta_{j}}$, the remaining function $g(s)$ is still a polynomial. Polynomials do not have finite poles, and the derivative $g^{'}(s)$ is also a polynomial. Therefore, $g(s)$ and $g^{'}(s)$ cannot have finite poles, $g(s_0)\neq\infty$ and $g^{'}(s_0)\neq\infty$.
\end{proof}
\begin{lemma}
Except for the zeros of (1.1), for all other points on the real axis, (2.2) holds.
\begin{equation}
TT1(s)=TT2(s).
\end{equation}

In which, $TT1(s)=(\pm)\mid{_{s=s_0}}\prod^{n}_{j=1}(s_0-p_{j})^{\beta_{j}}(\prod^{m}_{l=1}(s-z_{l})^{\gamma_{l}})^{'}-
(\pm)\prod^{n}_{j=1}(s-p_{j})^{\beta_{j}})^{'}
\prod^{m}_{l=1}(s_0-z_{l})^{\gamma_{l}}$.

$TT2(s)=((s-s_0)^{\gamma_0}g(s))^{'}
\prod^{m}_{l=1}(s_0-z_{l})^{\gamma_{l}}$.
\end{lemma}

\begin{proof}
Differentiate the equation: $(\pm)\mid{_{s=s_0}}\frac{\prod^{n}_{j=1}(s_0-p_{j})^{\beta_{j}}}
{\prod^{m}_{l=1}(s_0-z_{l})^{\gamma_{l}}}\prod^{m}_{l=1}(s-z_{l})^{\gamma_{l}}-
(\pm)\prod^{n}_{j=1}(s-p_{j})^{\beta_{j}}=(s-s_0)^{\gamma_0}g(s)$.

$(\pm)\mid{_{s=s_0}}\frac{\prod^{n}_{j=1}(s_0-p_{j})^{\beta_{j}}}
{\prod^{m}_{l=1}(s_0-z_{l})^{\gamma_{l}}}
(\prod^{m}_{l=1}((s-z_{l})^{\gamma_{l}}))^{'}-(\pm)
(\prod^{n}_{j=1}((s-p_{j})^{\beta_{j}}))^{'}
=((s-s_0)^{\gamma_0}g(s))^{'}$.

Since we exclude the zeros of (1.1), at the finite point $s_0$ on the real axis, the factor ${\prod^{m}_{l=1}(s_0-z_{l})^{\gamma_{l}}}$ of (1.1) is non-zero. Multiplying both sides by this factor ${\prod^{m}_{l=1}(s_0-z_{l})^{\gamma_{l}}}$ yields: $(\pm)\mid{_{s=s_0}}\prod^{n}_{j=1}(s_0-p_{j})^{\beta_{j}}(\prod^{m}_{l=1}(s-z_{l})^{\gamma_{l}})^{'}-
(\pm)(\prod^{n}_{j=1}(s-p_{j})^{\beta_{j}})^{'}
\prod^{m}_{l=1}(s_0-z_{l})^{\gamma_{l}}=(\gamma_0(s-s_0)^{\gamma_0-1}g(s)+(s-s_0)^{\gamma_0}g^{'}(s))
\prod^{m}_{l=1}(s_0-z_{l})^{\gamma_{l}}$.

Let $TT1(s)=(\pm)\mid{_{s=s_0}}\prod^{n}_{j=1}(s_0-p_{j})^{\beta_{j}}(\prod^{m}_{l=1}(s-z_{l})^{\gamma_{l}})^{'}-
(\pm)(\prod^{n}_{j=1}(s-p_{j})^{\beta_{j}})^{'}
\prod^{m}_{l=1}(s_0-z_{l})^{\gamma_{l}}$.

$TT2(s)=((s-s_0)^{\gamma_0}g(s))^{'}
\prod^{m}_{l=1}(s_0-z_{l})^{\gamma_{l}}$. By utilizing these two new expressions, we obtain (2.2).
\end{proof}

The breakaway points must lie on the root loci sharing the same unit complex value. Therefore, in $TT1(s)$, the sign $(\pm)$ before $\prod^{n}_{j=1}(s_0-p_{j})^{\beta_{j}}(\prod^{m}_{l=1}(s-z_{l})^{\gamma_{l}})^{'}$ is identical to the sign $(\pm)$ before $(\prod^{n}_{j=1}(s-p_{j})^{\beta_{j}})^{'}\prod^{m}_{l=1}(s_0-z_{l})^{\gamma_{l}}$.

\begin{lemma}
Let $s_0=(\sigma_0+it_0)\in\Delta $. If the real finite point $s_0$  satisfies equation (2.3),
\begin{equation}
\frac{dK(s)}{ds}=0.
\end{equation}
Then, $TT1(s_0)=0$ and $((s-s_0)^{\gamma_0}g(s))^{'}\mid{_{s=s_0}}=0$.
\end{lemma}

\begin{proof}
Substituting the real finite point $s_0$ into (2.3), according to the requirement of this lemma, we obtain:

$\frac{dK(s_0)}{ds}
=(\pm)\mid{_{s=s_0}}\frac{(\prod^{n}_{j=1}(s-p_{j})^{\beta_{j}})^{'}\mid{_{s=s_0}}
\prod^{m}_{l=1}(s_0-z_{l})^{\gamma_{l}}-
\prod^{n}_{j=1}(s_0-p_{j})^{\beta_{j}}
}{(\prod^{m}_{l=1}(s_0-z_{l})^{\gamma_{l}})^{2}}$

$\frac{(\prod^{m}_{l=1}((s-z_{l})^{\gamma_{l}}))^{'}\mid{_{s=s_0}}}{}=0$.

From $\frac{dK(s_0)}{ds}$, substituting $s_0$ into the left side of (2.2) yields: $TT1(s_0)=(\pm)\mid{_{s=s_0}}(\prod^{n}_{j=1}(s_0-p_{j})^{\beta_{j}}(\prod^{m}_{l=1}(s-z_{l})^{\gamma_{l}})^{'}\mid{_{s=s_0}}-
(\prod^{n}_{j=1}(s-p_{j})^{\beta_{j}})^{'}\mid{_{s=s_0}}
\prod^{m}_{l=1}(s_0-z_{l})^{\gamma_{l}})=0.$
Thus,  $TT2(s)=0$.

For the right side of (2.2), consider the expression: $TT2(s)=((s-s_0)^{\gamma_0}g(s))^{'}
\prod^{m}_{l=1}(s_0-z_{l})^{\gamma_{l}}$. Since the factor $\prod^{m}_{l=1}(s-z_{l})^{\gamma_{l}}$  has no finite pole, and $s_0$ is neither a zero nor a pole of the function $\prod^{m}_{l=1}(s-z_{l})^{\gamma_{l}}$. $TT2(s_0)=((s-s_0)^{\gamma_0}g(s))^{'}
\prod^{m}_{l=1}(s_0-z_{l})^{\gamma_{l}}=0$, we can remove  $\prod^{m}_{l=1}(s_0-z_{l})^{\gamma_{l}}$. This yields: $((s-s_0)^{\gamma_0}g(s))^{'}\mid{_{s=s_0}}=0$.
\end{proof}

\begin{lemma}
Let $s_0=(\sigma_0+it_0)\in\Delta $. If the real finite point $s_0$  satisfies equation (2.3), then $s_0$ is a real finite breakaway point of the root loci of (1.1).
\end{lemma}

\begin{proof}
We analyse three cases on the basis of $\gamma_0>0$:

1.  Case $\gamma_0=1$:

If $\gamma_0=1$,  $g^{'}(s_0)\neq\infty$, $(s_0-s_0)^{\gamma_0}g^{'}(s_0)
=0$. $g(s_0)\neq 0$. $(g(s_0)+(s_0-s_0)^{\gamma_0}g^{'}(s_0))=g(s_0)\neq 0$.

$((s-s_0)^{\gamma_0}g(s))^{'}\mid{_{s=s_0}}=(g(s_0)+(s_0-s_0)^{\gamma_0}g^{'}(s_0))
=g(s_0)\neq 0$. By Lemma 2.16, $((s-s_0)^{\gamma_0}g(s))^{'}\mid{_{s=s_0}}=0$. Therefore, this contradiction arises. Hence, $\gamma_0=1$ is invalid.

2. Case $1>\gamma_0>0$:

If $1>\gamma_0>0$,  $g^{'}(s_0)\neq\infty$, $(s_0-s_0)^{\gamma_0}g^{'}(s_0)
=0$. $\gamma_0(s_0-s_0)^{\gamma_0-1}=\infty$, $g(s_0)\neq 0$. $((s-s_0)^{\gamma_0}g(s))^{'}\mid{_{s=s_0}}=(\gamma_0(s_0-s_0)^{\gamma_0-1}g(s_0)+(s_0-s_0)^{\gamma_0}g^{'}(s_0))=\infty$.

$((s-s_0)^{\gamma_0}g(s))^{'}\mid{_{s=s_0}}=\infty$.

By Lemma 2.16, $((s-s_0)^{\gamma_0}g(s))^{'}\mid{_{s=s_0}}=0$. Therefore, this contradiction arises. Hence, $1>\gamma_0>0$ is invalid.

3. Case $\gamma_0>1$:

Because $\gamma_0>1$, $((s-s_0)^{\gamma_0}g(s))^{'}\mid{_{s=s_0}}=\gamma_0(s_0-s_0)^{\gamma_0-1}g(s_0)+(s_0-s_0)^{\gamma_0}g^{'}(s_0)=0$ holds. The left-hand side of (2.2) is equal to the right-hand side of (2.2). (2.2) holds. Hence, $\gamma_0>1$ is valid.

By Lemma 2.13, the roots of (1.1) and its characteristic equation are identical.
The left-hand side of the characteristic equation with gain $K(s_0)$ is rewritten as another expression: $(s-s_0)^{\gamma_0}g(s)$. $s_0$ is a root of (1.1).  $\gamma_0$ is a positive integer. Therefore, the condition $\gamma_0>1$ implies that the point $s_0$ is a multiple root of (1.1), and the finite point $s_0$ is the intersection point of at least two root loci on the real axis. The real finite point $s_0$ is a real finite breakaway point of (1.1).

Only $\gamma_0>1$ is valid, confirming that $s_0$ is a real finite breakaway point.
\end{proof}
\begin{lemma}
Let $s_0=(\sigma_0+it_0)\in\Delta $ be a point on the real axis. If $s_0$ is a real finite breakaway point of the root loci of (1.1), then $s_0$ must satisfy (2.3).
\end{lemma}
\begin{proof}
From (2.2), we obtain:

$(\pm)\mid{_{s=s_0}}\frac{\prod^{n}_{j=1}(s_0-p_{j})^{\beta_{j}}}
{\prod^{m}_{l=1}(s_0-z_{l})^{\gamma_{l}}}
(\prod^{m}_{l=1}((s-z_{l})^{\gamma_{l}}))^{'}-
(\pm)(\prod^{n}_{j=1}((s-p_{j})^{\beta_{j}}))^{'}
=\gamma_0(s-s_0)^{\gamma_0-1}g(s)+(s-s_0)^{\gamma_0}g^{'}(s)$. $g(s_0)\neq\infty$ and  $g^{'}(s_0)\neq\infty$.

Substituting $s_0$ into the right-hand side of the above equation, since $s_0$ is a real finite breakaway point of (1.1), $s_0$ is a multiple root of the characteristic equation. Therefore, $\gamma_0>1$. $\gamma_0(s_0-s_0)^{\gamma_0-1}g(s_0)+(s_0-s_0)^{\gamma_0}g^{'}(s_0)=0$. $TT2(s_0)=0$.
By Lemma 2.15, $TT1(s_0)=0$, we obtain: $(\pm)\mid{_{s=s_0}}(\prod^{n}_{j=1}(s_0-p_{j})^{\beta_{j}}
(\prod^{m}_{l=1}(s-z_{l})^{\gamma_{l}})^{'}\mid{_{s=s_0}}-
\prod^{m}_{l=1}(s_0-z_{l})^{\gamma_{l}}
(\prod^{n}_{j=1}(s-p_{j})^{\beta_{j}})^{'}\mid{_{s=s_0}})
=0$.
Differentiating the gain function and substituting $s_0$ into the derivative, we obtain:

$\frac{dK(s_0)}{ds}
=(\pm)\mid{_{s=s_0}}\frac{(\prod^{n}_{j=1}(s-p_{j})^{\beta_{j}})^{'}\mid{_{s=s_0}}
\prod^{m}_{l=1}(s_0-z_{l})^{\gamma_{l}}-
\prod^{n}_{j=1}(s_0-p_{j})^{\beta_{j}}
}{(\prod^{m}_{l=1}(s_0-z_{l})^{\gamma_{l}})^{2}}$

$\frac{(\prod^{m}_{l=1}((s-z_{l})^{\gamma_{l}}))^{'}\mid{_{s=s_0}}}{}
=0$.
\end{proof}

Lemma 2.17 proves the sufficient condition for breakaway points, whereas Lemma 2.18 proves the necessary condition. Combining these two lemmas, we obtain the following sufficient and necessary conditions for breakaway points.

\begin{theorem}
Let $s_0=(\sigma_0+it_0)\in\Delta $. The real finite point $s_0$ is a real finite breakaway point of the root loci of (1.1) if and only if the real finite point $s_0$ satisfies (2.3).
\end{theorem}

When the unit complex values of the root loci are distinct, these root loci cannot intersect. When the unit complex values of the root loci are identical, (under these conditions)the root loci may intersect. The intersection points at real finite locations are real finite breakaway points. These intersecting root loci separate from these real finite breakaway points again. Multiple zeros and poles of the root locus equation cannot exhibit this property. Thus, the real finite breakaway points of the root loci are distinct from the real multiple zeros and poles of (1.1).

On both sides of each non-standard real breakaway point on the real axis, the monotonicity of the gain remains unchanged. On both sides of each standard real breakaway point on the real axis, the monotonicity of the gain must change.

\begin{definition}
The set $\Omega_c$ consists of all real critical points of $RF$, excluding multiple zeros and multiple poles of $RF$.
\end{definition}

\begin{definition}
The set $\Omega_b$ consists of all real breakaway points of the root loci of $RF$, excluding multiple zeros and multiple poles of $RF$.
\end{definition}
\begin{definition}
The set $\Omega_K$ consists of all real critical points of $K(s)$, excluding the multiple zeros and multiple poles of $RF(s)$.
\end{definition}

\begin{theorem}
Take $s=(\sigma+it)\in\Phi$.

1. If the point $cp \in \Omega_c$, then $cp \in \Omega_K$.

2. If the point $gp \in \Omega_K$, then $gp \in \Omega_c$.

3. The point $bp \in \Omega_b$ if and only if $bp \in \Omega_c$ and $bp \in \Omega_K$.
\end{theorem}
\begin{proof}
$K(s)=\pm \frac{\prod_{j=1}^{n} (s-p_{j})^{\beta_j}}{\prod_{i=1}^{m} (s-z_{i})^{\gamma_l}}$.

$
\frac{dK(s)}{ds}
=\pm \frac{(\prod^{n}_{j=1}(s-p_{j})^{\beta_j})^{'}
\prod^{m}_{l=1}(s-z_{l})^{\gamma_l}-
\prod^{n}_{j=1}(s-p_{j})^{\beta_j}(\prod^{m}_{l=1}(s-z_{l})^{\gamma_l})^{'}
}{(\prod^{m}_{l=1}(s-z_{l})^{\gamma_l})^{2}}
$

$
(RF(s))^{'}
=\frac{(\prod^{m}_{l=1} (s-z_l)^{\gamma_l})^{'}\prod^{n}_{j=1}  (s-p_j)^{\beta_j}-\prod^{m}_{l=1} (s-z_l)^{\gamma_l}(\prod^{n}_{j=1}  (s-p_j)^{\beta_j})^{'}}{(\prod^{n}_{j=1}  (s-p_j)^{\beta_j})^2}
$

The multiple zeros and multiple poles of $RF(s)$ are zeros of the numerator polynomials of $\frac{dK(s)}{ds}$ and $(RF(s))^{'}$. Specifically:

1. The multiple zeros of $RF(s)$ are zeros of $(RF(s))^{'}$, but poles of $\frac{dK(s)}{ds}$.

2. The multiple poles of $RF(s)$ are zeros of $\frac{dK(s)}{ds}$, but are still poles of $(RF(s))^{'}$.

Therefore, multiple zeros and multiple poles of $RF(s)$ cannot be the critical points of both $RF(s)$ and $K(s)$. Thus, when we study critical points of both $RF(s)$ and $K(s)$, we need to exclude multiple zeros and multiple poles of $RF(s)$.

Excluding the multiple zeros and multiple poles of $RF(s)$, the zeros of $\frac{dK(s)}{ds}$ and $\frac{dRF(s)}{ds}$ are identical. If $cp \in \Omega_c$, then $cp \in \Omega_K$. Conversely, if $gp \in \Omega_K$, then $gp \in \Omega_c$.

By Lemma 2.19, all real finite breakaway points of (1.1) are real finite critical points of $K(s)$ and $RF(s)$. If $bp \in \Omega_b$, then $bp \in \Omega_c$ and $bp \in \Omega_K$. Conversely,  all real finite critical points of $K(s)$ and $RF(s)$ are real finite breakaway points of (1.1). If $bp \in \Omega_c$ and $bp \in \Omega_K$, then $bp \in \Omega_b$.
\end{proof}

\begin{theorem}
No two segments of root loci of any degree of (1.1)  overlap in the finite region of $\mathbb{C}\cup\{\infty\}$.
\end{theorem}

\begin{proof}
The finite breakaway point is the intersection point of the root loci of the root locus equation (1.1). An overlapping segment on the real axis would consist of a continuous set of intersection points of the root locus of (1.1). If such overlapping segments exist, they would be special intersection points, specifically an infinite number of continuous intersection points. To compute the finite breakaway points and prove their existence, we use the formula for computing finite breakaway points in Lemma 2.4. Lemma 2.4 provides a necessary  condition for computing the finite breakaway point:

$\frac{dK(s)}{ds}
=\pm \frac{(\prod^{n}_{j=1}(s-p_{j})^{\beta_{j}})^{'}
\prod^{m}_{l=1}(s-z_{l})^{\gamma_{l}}-
\prod^{n}_{j=1}(s-p_{j})^{\beta_{j}}(\prod^{m}_{l=1}(s-z_{l})^{\gamma_{l}})^{'}
}{(\prod^{m}_{l=1}(s-z_{l})^{\gamma_{l}})^{2}}=0.$

All finite breakaway points are finite zeros of the derivative of the function $RF(s)$. The derivative of $RF(s)$ is a rational function that has a finite number of finite zeros. These finite zeros must be isolated in $\mathbb{C}\cup\{\infty\}$; they cannot be continuously distributed. Therefore, these finite zeros cannot form a continuous curve. On the basis of the previous proof, any segment of all the root loci of an arbitrary degree number cannot overlap.
\end{proof}

\section{Setting cases when Shapiro's Conjecture 12 holds}
All lemmas and theorems in this section are original to the authors and do not rely on prior literature.

In the following proof, the concept of the root locus extension is essential. Since the gain values of the root locus of (1.1) increase strictly monotonically from $0$ at the poles to positive infinity at the zeros, the root locus extension of (1.1) implies that the root locus moves from points with small gains to points with large gains. Alternatively, the root locus extension means that a root locus moves from a point closer to a pole to a point closer to a zero.

The root loci originating from $\mathbb{C}^{*}$ may intersect at the real breakaway points on the real axis. However, these root loci do not enter the real axis; instead, they immediately return to $\mathbb{C}^{*}$. Therefore, these root loci in $\mathbb{C}^{*}$ do not affect the root loci on the real axis. The root loci on the real axis extend continuously  across such real breakaway points. When the root loci on the real axis extend from one side of the real breakaway point to the other side, the gains of the root loci on the real axis are continuous and monotonic. We refer to these points as non-standard real breakaway points. In subsequent proofs, discussions of real breakaway points and their intersections on the real axis will exclude these non-standard real breakaway points and these intersection root loci.

All other real breakaway points are called standard real breakaway points. Their defining property is that two root loci exist on either side of the real breakaway point, with opposite extension directions:

1. Both two root loci either extend toward the real breakaway point, or

2. Both two root loci extend away from the real breakaway point.

This paper focuses on the real roots of polynomial $\Delta$. Accordingly, in this section, our study is restricted to the $2q\pi$-degree($q\in Z$) root loci of (1.1); root loci of other degrees are not considered.

\begin{lemma}
On the root loci of $2q\pi$ degree of (1.1), the standard real breakaway points of the root loci of the $2q\pi$ degree of (1.1) are the extreme points of the gain of (1.1).
\end{lemma}
\begin{proof}
Since the $2q\pi$-degree root loci of (1.1) are symmetric about the real axis, if these root loci intersect at a standard real breakaway point $b_1$ from $\mathbb{C}^{*}$, the two root loci depart from the real breakaway point $b_1$ on the real axis. By Theorem 2.24, the two root loci cannot extend in the same direction, and they must extend in opposite directions along the real axis.

Upon departure from the standard real breakaway point $b_1$, on the root loci which extend in opposite directions along the real axis, by the definition of  the root locus extension, the monotonicity of the gain $K(x)$ of $2q\pi$-degree root loci of (1.1) must be opposite. Therefore, the standard real breakaway point $b_1$ serves as a boundary point where the gain's monotonicity changes, implying that the gain $K(x)$ of (1.1) attains an extremum at the boundary point where the gain's monotonicity is opposite.

Similarly, when two root loci on the real axis extend toward each other, intersect at a standard real breakaway point $b_2$, and then leave the real axis to enter $\mathbb{C}^{*}$, by the definition of  the root locus extension, the gain's monotonicity on either side of the standard real breakaway point is again opposite. This kind of standard real breakaway point is also a boundary where the gain's monotonicity reverses. At the boundary point where the gain's monotonicity is opposite, the gain of (1.1) reaches an extremum there.

Thus, each standard real breakaway point is an extremum of the gain of (1.1).
\end{proof}
By Theorem 2.23, at each standard real breakaway point $b_i$ of (1.1), $\frac{dK(s)}{ds}\mid_{s=b_i}=0$. $i=1,2$. It implies that $b_i$ are the extreme points of the gain function of (1.1). On the root loci of $2q\pi$ degree, the gain function of (1.1) is same as gain of (1.1).

The gain expression of (1.2) is given by: $K=\abs{\frac{(p^{'})^2}{p^{''}p}}$. Substituting $s=\pm\infty$ into the gain expression, we obtain $K_{\pm\infty}=\abs{\frac{(p^ {'}(\pm\infty))^2}{p^{''}(\pm\infty)p(\pm\infty)}}=\frac{n}{n-1}$. Thus, $K_{\pm\infty}=K_0$.

\begin{lemma}
The gains at the two infinity points of the real axis are given by $K_{\pm\infty}=\frac{n}{n-1}$. The two infinity points of the real axis are roots of $\Delta$.
\end{lemma}
For any polynomial $p$, its first derivative $p^{'}$, and its second derivative $p^{''}$, the two points at infinity on the real axis are roots of  $\Delta$. In this case, Shapiro's conjecture 12 must always hold, making the study of this conjecture trivial. Therefore, the roots at infinity are not the roots of the polynomial $\Delta$ that Shapiro's conjecture 12 intends to study. In this paper, we exclude the roots at infinity of $\Delta$.
\begin{lemma}
When $PP\in \Gamma_{11}\cup\Gamma_{121}\cup\Gamma_{122}$, then $K(x)$ is continuous on the entire real axis. The real axis consists of  $2q\pi$-degree root loci.
\end{lemma}
\begin{proof}
When $PP\in \Gamma_{11}\cup\Gamma_{121}\cup\Gamma_{122}$, the polynomial $p$ has no real zeros, and $p^{''}$ has no real zeros. The gain function of (1.2) is given by: $K(x)=\frac{(p^{'}(x))^2}{p^{''}(x)p(x)}$. $K(x)$ has no real poles. $K(x)$ is continuous on the entire real axis. (1.2) has no real zeros. $p^{'}$ has only one simple real zero $p_0$.  (1.2) has only one second-order real pole $p_0$. Therefore, the real axis consists of  $2q\pi$-degree root loci. In which there exist $2q\pi$-degree two root locus originating from $p_0$.
\end{proof}

$p^{''}$ has no real zeros. The gains strictly and monotonically increase from $K(p_0)=0$ to $K_{\pm\infty} = \frac{n}{n-1}$.
$K(x)<K_0$, for $x\in(-\infty, p_0)$ or $x\in(p_0, +\infty)$. Hence no real roots of $\Delta$.

\begin{lemma}
When $PP\in \Gamma_{11}$, then $\sharp_r \Delta=0$.
\end{lemma}
\begin{proof}
When $PP\in \Gamma_{11}$, by Lemma 2.3, the interval $(p_0, +\infty)$ is a part of the complete root locus of $2q\pi$ degree, originating from $p_0$. The interval $(-\infty, p_0)$ is a part of the complete root locus of $2q\pi$ degree, originating from $p_0$. Since there is no standard real breakaway point of the root loci of (1.2) on the real axis.
By Theorem 2.8: On each root locus of (1.2), when the point $s$ moves from the poles of (1.2) to the zeros of (1.2), then gains are strictly and monotonically increasing.

At the positive and negative infinity points $\pm\infty$ on the real axis, the gain attains its maximum value, $K_{\pm\infty}=\frac{n}{n-1}$. The gains  strictly and monotonically increase from $K(p_0)=0$ to $K_{\pm\infty} = \frac{n}{n-1}$.

Therefore, in the two infinite intervals $(-\infty, p_0)$ and $(p_0, +\infty)$, the gain at every point is strictly less than $K_0$. $K(x)<K_0$, for $x\in(-\infty, p_0)$ or $x\in(p_0, +\infty)$.

This implies that there is no point in these two infinite intervals where the gain equals $K_0$, and hence no real roots of  $\Delta$. We conclude that if $PP\in \Gamma_{11}$, then $\sharp_r \Delta=0$.
\end{proof}
\begin{lemma}
When $PP\in \Gamma_{121}\cup\Gamma_{122}$, then (1.2) must have maximum real breakaway points.
\end{lemma}
\begin{proof}
When $PP\in \Gamma_{121}\cup\Gamma_{122}$, the polynomial $p$ has no real zeros, and $p^{''}$ also has no real zeros. Since $p^{'}$ has only one simple real zero $p_0$, (1.2) has a single second-order real pole at $p_0$. The real axis consists of $2q\pi$-degree root loci. In which there exist $2q\pi$-degree two root locus both originating from $p_0$.

According to the requirement that $\Gamma_{121}\cup\Gamma_{122}$ satisfies, (1.2) has a standard real breakaway point, so there must exist root loci in $\mathbb{C}^{*}$ that enter the real axis. Given that the $2q\pi$-degree root loci of (1.2) are symmetric about the real axis, at least two root loci in $\mathbb{C}^{*}$ must first intersect at a standard real breakaway point $b_{11}$ on the real axis. These two root loci, denoted as $L_{11}$ and $L_{12}$, enter the real axis and then separate. $L_{11}$ and $L_{12}$ are two distinct root loci, by Theorem 2.24, they cannot overlap. $L_{11}$ and $L_{12}$ cannot extend in the same direction. Therefore, $L_{11}$ and $L_{12}$ must extend in opposite directions along the real axis.

Since (1.2) has no zeros on the real axis, neither $L_{11}$ nor $L_{12}$ can end at zeros of (1.2). Instead, they must leave their respective intervals on the real axis. To do so, $L_{11}$ and $L_{12}$ must intersect with another root loci before departing from the real axis and re-entering $\mathbb{C}^{*}$. Because at least one root locus originates from $p_0$ on the real axis, either $L_{11}$ or $L_{12}$ must intersect with another root locus, which we denote as $J_{11}$. This intersection generates another standard real breakaway point $b_{12}$.

On either side of the standard real breakaway point $b_{11}$, the root loci $L_{11}$ and $L_{12}$ extend in opposite directions. One of them must extend in the positive direction of the real axis. Without loss of generality, assume $L_{11}$ extends in the positive direction of the real axis, departing from $b_{11}$. Along  $L_{11}$, the gains strictly monotonically increase. The other must extend in the negative direction of the real axis. $L_{12}$ extends in the negative direction of the real axis, departing from $b_{11}$, and the gain strictly monotonically decreases along $L_{12}$. Thus, to the right of $b_{11}$, the gain monotonically increases, while to the left, it monotonically decreases. Therefore, the gain attains its minimum value at $b_{11}$.

Now, let $J_{12}$ be the root locus that intersects with $J_{11}$ at the standard real breakaway point $b_{12}$. Here, $J_{12}$ is either $L_{11}$ or $L_{12}$. $J_{11}$ and $J_{12}$ extend toward each other and intersect at $b_{12}$. One of them must extend in the positive direction of the real axis. Suppose $J_{11}$ extends in the positive direction of the real axis, approaching $b_{12}$, with the gain strictly monotonically increasing along $J_{11}$.  The other must extend in the negative direction of the real axis. $J_{12}$ extends in the negative direction of the real axis, approaching $b_{12}$, with the gains strictly monotonically decreasing along $J_{12}$. Consequently, to the left of $b_{12}$, the gains monotonically increase, while to the right of $b_{12}$, it monotonically decreases. Thus, the gain attains its maximum value at $b_{12}$.

In conclusion, when $PP\in \Gamma_{121}\cup\Gamma_{122}$, (1.2) must have maximum real breakaway points.
\end{proof}
\begin{lemma}
When $PP\in \Gamma_{121}$, then $\sharp_r \Delta=0$.
\end{lemma}
\begin{proof}
At the endpoints $\pm\infty$ of the real axis, the gain is $K_{\pm\infty}=\frac{n}{n-1}$. At the point $p_0$, the gain is $K(p_0)=0$. If the gains at all maximum real breakaway points are less than $K_0$, $K(x)$ is continuous on the entire real axis, when $PP\in \Gamma_{121}$, taking all positive real values between $0$ and $\frac{n}{n-1}$. By the Intermediate Value Theorem, $K(x)<K_0$ for all $x$ on the real axis. Consequently, no point on the real axis has a gain of $K_0$, and thus $\Delta$ has no real roots. Therefore, if $PP\in \Gamma_{121}$, then $\sharp_r \Delta=0$.
\end{proof}
\begin{lemma}
If $PP\in \Gamma_{122}$, then $\sharp_r \Delta>0$.
\end{lemma}
\begin{proof}
By definition of $\Gamma_{122}$, there exists at least one maximum real breakaway point $b_{122}$ such that $K(b_{122})\ge K_0$. Because $K(x)$ is continuous on the entire real axis, it takes all positive real values between $K(b_{122})$ and 0. The Intermediate Value Theorem guarantees the existence of at least one point $x_{*}$ such that $K(x_{*})=K_0$. This $x_{*}$ is a real root of $\Delta$. Therefore, if $PP\in \Gamma_{122}$, then $\sharp_r \Delta>0$.
\end{proof}

\begin{lemma}
When $PP\in \Gamma_{22}$, then $\sharp_r \Delta>0$.
\end{lemma}
\begin{proof}
When $PP\in \Gamma_{22}$, the polynomial $p$ has no real zeros, and $p^{'}$ has only one real zero $p_0$. To the right of $p_0$, $p^{''}$ possesses no real zeros, while to the left of $p_0$, $p^{''}$ has real zeros.

Let $z_3$ be the real zero of (1.2) adjacent to $p_0$ on the left. To the right of $p_0$, (1.2) has no real zeros. As $p_0$ is a second-order pole. To the right side of the interval $(z_3,p_0)$, (1.2) has an even number of real zeros and poles. By Lemma 2.3, the interval $(z_3,p_0)$ is a $2q\pi$-degree root locus of (1.2), extending from the pole $p_0$ to the zero $z_3$. At the point $z_3$, the gain $K(z_3)=+\infty$, while at $p_0$, $K(p_0)=0$. Since $K(x)$ is continuous on $(z_3,p_0)$(there are no pole of $K(x)$ in $(z_3,p_0)$). $K(x)$ attains all positive real values from $0$ to $+\infty$.  By the Intermediate Value Theorem, there exists at least one point $x_{*} \in (z_3, p_0)$ such that $K(x_{*}) = K_0$ in the interval $(z_3,p_0)$. This $x_{*}$ is a real root of $\Delta$ on the left of $p_0$. Thus, when $PP\in \Gamma_{22}$, then $\sharp_r \Delta>0$.
\end{proof}

There may exist real breakaway points of (1.2) in the interval $(z_3,p_0)$. This implies that $(z_3,p_0)$ contains multiple root loci rather than a single root locus. In the proof of Lemma 3.8, we only provided the proof that the interval $(z_3,p_0)$ is a single root locus. If there exist real breakaway points of (1.2) in the interval $(z_3,p_0)$. In the interval $(z_3,p_0)$, $K(x)$ still attains all positive real values from $0$ to $+\infty$. This doesn't affect the proof and result of Lemma 3.8.

\begin{lemma}
In the interval $(z_m,+\infty)$, if (1.2) has a standard real breakaway point, then  (1.2) must have minimum real breakaway points.
\end{lemma}
\begin{proof}
In $(z_m, +\infty)$, $p$ has no real zeros and $p^{''}$ also has no real zeros. To the right of $z_m$, there exists a standard real breakaway point of (1.2). Consequently, there must be root loci in $\mathbb{C}^{*}$ that enter $(z_m,+\infty)$. Since the $2q\pi$-degree root loci of (1.2) are symmetric about the real axis, at least two root loci in $\mathbb{C}^{*}$ must first intersect at a standard real breakaway point $b_{21}$ in $(z_m,+\infty)$. These two root loci, denoted as $L_{21}$ and $L_{22}$, enter the real axis and then separate. $L_{21}$ and $L_{22}$ are two distinct root loci, by Theorem 2.24, they cannot overlap. $L_{21}$ and $L_{22}$ cannot extend to the same direction. Therefore, $L_{21}$ and $L_{22}$ must  extend in opposite directions along the real axis.

On either side of $b_{21}$, the root loci $L_{21}$ and $L_{22}$ extend in opposite directions. One of them must extend in the positive direction of the real axis. Without loss of generality, assume that $L_{21}$ extends in the positive direction of the real axis, departing from $b_{21}$. Along $L_{21}$, the gain strictly monotonically increases. The other must extend in the negative direction of the real axis. $L_{22}$ extends in the negative direction of the real axis, departing from $b_{21}$, and the gain strictly monotonically decreases along $L_{22}$. Thus, to the right of $b_{21}$, the gain monotonically increases, while to the left, it monotonically decreases. Therefore, the gain attains its minimum value at $b_{21}$.

In conclusion, in $(z_m,+\infty)$, if $PP$ has a standard real breakaway point. Then, (1.2) must have minimum real breakaway points.
\end{proof}
By repeating the proof of Lemma 3.9, we can establish Lemma 3.10.
\begin{lemma}
In the interval $(-\infty, z_s)$, if $PP$ has a standard real breakaway point. Then, (1.2) must have minimum real breakaway points.
\end{lemma}
\begin{lemma}
In a finite interval $(z_1, z_2)$ between any two adjacent real zeros of $p^{''}$, (1.2) must have minimum real breakaway points.
\end{lemma}
\begin{proof}
Consider a finite interval $(z_1, z_2)$ between any two adjacent real zeros of $p^{''}$. In this interval:

1. $p$ has no real zeros,

2. $p^{''}$ has no real zeros (by definition, since $z_1$ and $z_2$ are consecutive zeros),

3. $z_1$ and $z_2$ are zeros of (1.2), and

4. There are no poles of (1.2) between $z_1$ and $z_2$.

Since both $z_1$ and $z_2$ must receive root loci, there must exist root loci in $\mathbb{C}^{*}$ that enter the real axis at some point within $(z_1, z_2)$.

The root loci of (1.2) on the real axis consist of both $2q\pi$ and $(2q\pi+\pi)$-degree branches. Due to the symmetry of these root loci about the real axis, at least two root loci in $\mathbb{C}^{*}$ must intersect at a standard real breakaway point $b_{31}$ on the real axis. Let $L_{31}$ and $L_{32}$ denote these two root loci as they enter the real axis. After entering, they separate. $L_{31}$ and $L_{32}$ are two distinct root loci, by Theorem 2.24, they cannot overlap. $L_{31}$ and $L_{32}$ cannot extend to the same direction. Therefore, $L_{31}$ and $L_{32}$ must  extend in opposite directions along the real axis.

On both sides of the standard real breakaway point $b_{31}$, the root loci $L_{31}$ and $L_{32}$ extend in opposite directions. One of them must extend in the positive direction of the real axis. Here, let $L_{31}$ extend in the positive direction of the real axis. $L_{31}$ departs from $b_{31}$. Therefore, the gain strictly monotonically increases along  $L_{31}$. The other root locus $L_{32}$ extends in the negative direction of the real axis. $L_{32}$ departs from $b_{31}$. The gain strictly monotonically decreases along $L_{32}$. On the real axis to the right of point $b_{31}$, the gain monotonically increases. On the real axis to the left of point $b_{31}$, the gain monotonically decreases. Thus, at the standard real breakaway point $b_{31}$, the gain attains a minimum value.

Therefore, in $(z_1, z_2)$, $PP$  has standard real breakaway point. (1.2) must have minimum real breakaway points.
\end{proof}
\begin{lemma}
In $(z_m,+\infty)$, $(-\infty, z_s)$ and the finite interval $(z_1, z_2)$ between any two adjacent real zeros of $p^{''}$, $K(x)$ is continuous.
\end{lemma}
\begin{proof}
In $(z_m,+\infty)$, $(-\infty, z_s)$ and the finite interval $(z_1, z_2)$ between any two adjacent real zeros of $p^{''}$, $p$ has no real zeros. $p^{''}$ has no real zeros. Then, $K(x)$ has no real poles. Thus, $K(x)$ is continuous in $(z_m,+\infty)$, $(-\infty, z_s)$ and the finite interval $(z_1, z_2)$.
\end{proof}

\begin{lemma}
For $PP\in \Gamma_{21}\cup \Gamma_{23}$, if either:

1. There is no standard real breakaway point in $(z_m,+\infty)$, or

2. All minimum real breakaway points in $(z_m,+\infty)$ have gains greater than $K_0$,

then $(z_m,+\infty)$ contains no root of $\Delta$.
\end{lemma}
\begin{proof}
At the endpoint $z_m$, the gain $K(z_m)=+\infty$. At the endpoint of positive infinity: $K_{+\infty}=\frac{n}{n-1}$. If there is no standard real breakaway point in $(z_m,+\infty)$, then $(z_m,+\infty)$ is a segment of the complete root locus of (1.2), with gains obtain all positive real values from $+\infty$ at zero $z_m$ to the $K_{+\infty}=\frac{n}{n-1}$ at the positive infinity. Since the gains of all points in $(z_m,+\infty)$  are larger than $K_0$, no point in $(z_m,+\infty)$ satisfies $K(x)=K_0$. So, $\Delta$ has no root in $(z_m,+\infty)$.

In the interval $(z_m,+\infty)$, if the standard real breakaway points exist, and all minimum real breakaway points in $(z_m,+\infty)$ have gains greater than $K_0$, then since $K(x)$ is continuous and attains all positive real values from $+\infty$ to $K_0$. So the gains of all points in $(z_m,+\infty)$  are larger than $K_0$, no point in $(z_m,+\infty)$ satisfies $K(x)=K_0$. So, $\Delta$ has no root in $(z_m,+\infty)$.
\end{proof}

\begin{lemma}
For $PP\in \Gamma_{21}\cup \Gamma_{23}$, in the interval $(z_m,+\infty)$, there exist the standard real breakaway points. If there exists at least one minimum real breakaway point $b_m$, $K(b_m)\le K_0$, then  in $(z_m,+\infty)$,  $\Delta$ must have a root.
\end{lemma}
\begin{proof}
At the zero $z_m$, the gain is $+\infty$, while at positive infinity the gain is $K_{+\infty}=\frac{n}{n-1}$. In the interval $(z_m,+\infty)$, if a standard real breakaway point exists, and there is at least one minimum real breakaway point $b_m$, $K(b_m)\le K_0$, then since $K(x)$ is continuous and attains all positive real values from $+\infty$ to $K(b_m)$, by the Intermediate Value Theorem, there must exist at least one point $x_{*}$ where $K(x_{*})=K_0$. This implies  $\Delta$ must have a root in $(z_m,+\infty)$.
\end{proof}

By repeating the proof of Lemma 3.13, we can establish Lemma 3.15.
\begin{lemma}
For $PP\in \Gamma_{23}$, if either:

1. no standard real breakaway point exists in $(-\infty, z_s)$, or

2. all minimum real breakaway points in $(-\infty, z_s)$ have gains greater than $K_0$,

then the polynomial $\Delta$ has no roots in $(-\infty, z_s)$.
\end{lemma}
Similarly, Lemma 3.16 follows from the proof method of Lemma 3.14:
\begin{lemma}
For $PP\in \Gamma_{23}$, in the left infinite interval $(-\infty, z_s)$ of $2q\pi$ degree, if a standard real breakaway point exists, and there exists at least one minimum real breakaway point $b_s$ such that $K(b_s)\le K_0$, then the polynomial $\Delta$ must have roots in $(-\infty, z_s)$.
\end{lemma}

\begin{lemma}
For $PP \in \Gamma_{21} \cup \Gamma_{23}$, in the finite interval $(z_1, z_2)$ of $2q\pi$ degree, if all minimum real breakaway points have gains greater than $K_0$, then the polynomial $\Delta$ has no roots in $(z_1, z_2)$.
\end{lemma}
\begin{proof}
At the endpoints $z_1$ and $z_2$ of $(z_1, z_2)$, the gain is $+\infty$. If all minimum real breakaway points have gains greater than $K_0$, $K(x)$ is continuous in $(z_1, z_2)$. Then,  in the finite interval $(z_1, z_2)$ of $2q\pi$ degree, $K(x)$ attains all positive real values from $+\infty$ to a value greater than $K_0$, so the gain at all points never equals $K_0$. By the Intermediate Value Theorem, $\Delta$  has no roots in this interval.
\end{proof}

\begin{lemma}
For $PP\in \Gamma_{21}\cup \Gamma_{23}$, in the finite interval $(z_1, z_2)$ of $2q\pi$ degree, if there exists at least one minimum real breakaway point $b_z$ such that $K(b_z)\le K_0$, then the polynomial $\Delta$ must have a root in $(z_1, z_2)$.
\end{lemma}
\begin{proof}
In the finite interval $(z_1, z_2)$ of $2q\pi$ degree, $K(x)$ is continuous.
At the endpoints $z_1$ and $z_2$, the gain is $+\infty$. If there exists at least one minimum real breakaway point $b_z$ such that $K(b_z)\le K_0$, then $K(x)$  takes all positive real values from $+\infty$ to $K(b_z)$. By the Intermediate Value Theorem, there must exist at least one point $x_{*}$ where $K(x_{*})=K_0$, ensuring the existence of a root of $\Delta$ in $(z_1, z_2)$.
\end{proof}

\begin{lemma}
When $PP\in \Gamma_{2121}\cup\Gamma_{2122}$, let $z_4$ denote the real zero of $p^{''}$ adjacent to $p_0$ and to the right of $p_0$. In the interval $(p_0, z_4)$ and $(-\infty, p_0)$, then there are no roots of $\Delta$.
\end{lemma}
\begin{proof}
To the right of the interval $(p_0,z_4)$, (1.2) has an odd number of real zeros. By Lemma 2.3, the interval $(p_0,z_4)$ is a $2q\pi+\pi$-degree root locus of (1.2), extending from a pole to a zero. Since the roots of $\Delta$ lie on the $2q\pi$-degree root locus, $(p_0,z_4)$ cannot contain any roots of $\Delta$.

To the right of $p_0$, (1.2) has an odd number of real zeros. As $p_0$ is a second-order pole. To the right side of the interval $(-\infty,p_0)$, (1.2) has an odd number of real zeros and poles. To the left of $p_0$, $p^{''}$ has no real zeros. By Lemma 2.3, $(-\infty,p_0)$ is a $2q\pi+\pi$-degree root locus emitted from the pole $p_0$. Again, since the roots of $\Delta$ lie on the $2q\pi$-degree root locus, $(-\infty,p_0)$ cannot contain any roots of $\Delta$.
\end{proof}

There may exist real breakaway points of (1.2) in the intervals $(p_0, z_4)$ and $(-\infty, p_0)$. This implies that each interval contains multiple root loci rather than a single root locus. In the proof of Lemma 3.19, we only provided the proof that the intervals $(p_0, z_4)$ and $(-\infty, p_0)$ are single root loci. If there exist real breakaway points of (1.2) in the intervals $(p_0, z_4)$ and $(-\infty, p_0)$. The intervals $(p_0, z_4)$ and $(-\infty, p_0)$ still $2q\pi+\pi$-degree root loci of (1.2). This doesn't affect the proof and result of Lemma 3.19.

When $PP\in \Gamma_{2121}\cup\Gamma_{2122}$, $p^{''}$ has real zeros to the right of $p_0$. Equation (1.2) has exactly one real pole $p_0$ but multiple real zeros. Therefore, the interval $(z_m,+\infty)$ must exist. These real zeros may form finite intervals of $2q\pi$ degree. Additionally, there exists a real zero $z_4$ of (1.2) adjacent to $p_0$, meaning the interval $(p_0,z_4)$ must also exist.

The interval $(z_m,+\infty)$ may or may not contain standard real breakaway points and constitute the first type of interval. The intervals between real zeros of (1.2) necessarily contain minimum real breakaway points, while the finite intervals of $2q\pi$ degree constitute the second type of interval. The third type consists of intervals $(p_0,z_4)$ extending from the poles of (1.2) to its zeros. Therefore, in the proofs of results of rational functions in $\Gamma_{2121}$ and $\Gamma_{2122}$, all these three types of intervals need to be considered.

For the proofs of results concerning rational functions in $\Gamma_{2121}$ and $\Gamma_{2122}$, all three types of intervals must be considered. Since only these three types of intervals exist to the right of $p_0$, the proofs are restricted to these three cases.
\begin{lemma}
When $PP\in \Gamma_{2121}$, then $\sharp_r \Delta=0$.
\end{lemma}
\begin{proof}
When $PP\in \Gamma_{2121}$, $p$ has no real zeros, and $p^{'}$ has exactly one real zero $p_0$. To the right of the real zero $p_0$, $p^{''}$  has an odd number of real zeros. Consequently, (2.1) has a real pole $p_0$. To the right of $p_0$, (2.1) has an odd number of real zeros. Let $z_4$ denote the real zero of $p^{''}$ adjacent to $p_0$ and to the right of $p_0$.

In the right infinite interval $(z_m,+\infty)$, there are no standard real breakaway points. Or, the gains of all minimum real breakaway points are greater than $K_0$. On the real axis, if finite intervals of $2q\pi$ degree exist to the right of $p_0$, in all such finite intervals, all minimum real breakaway points have gains greater than $K_0$. By Lemma 3.13 and Lemma 3.17, this implies:

1. $(z_m,+\infty)$ contains no roots of $\Delta$.

2. All $2q\pi$-degree finite intervals contain no roots of $\Delta$.

Thus, based on the two results, in the infinite interval $(z_4,+\infty)$ to the right of $z_4$, $\Delta$ has no roots.

By Lemma 3.19,  $(p_0,z_4)$  and $(-\infty,p_0)$ cannot contain any roots of $\Delta$.

Combining these results, we conclude that when $PP\in \Gamma_{2121}$, $\sharp_r \Delta=0$.
\end{proof}

\begin{lemma}
When $PP\in \Gamma_{2122}$, then $\sharp_r \Delta>0$.
\end{lemma}
\begin{proof}
When $PP\in \Gamma_{2122}$, $p$ has no real zeros, $p^{'}$ has  one real zero $p_0$. To the right of $p_0$, $p^{''}$ has an odd number of real zeros. (1.2) has a real pole $p_0$, and to its right, (1.2) has an odd number of real zeros. Let $z_4$ denote the real zero of $p^{''}$ adjacent to $p_0$ and to the right of $p_0$.

If there exists at least one minimum real breakaway point $b_{2122}$ in $(z_m,+\infty)$ such that $K(b_{2122})\le K_0$, then by Lemma 3.14, $(z_m,+\infty)$ contains at least one root of $\Delta$. Thus, $\sharp_r \Delta>0$.

If a finite interval of $2q\pi$ degree exists to the right of $p_0$, and within at least one such finite interval, there is at least one minimum real breakaway point $b_{2122}$ such that $K(b_{2122})\le K_0$, then by Lemma 3.18, this interval contains at least one root of $\Delta$. Hence, $\sharp_r \Delta>0$.

By Lemma 3.19,  $(p_0,z_4)$  and $(-\infty,p_0)$ cannot contain any roots of $\Delta$.

Combining these results, we conclude that when $PP\in \Gamma_{2122}$, $\Delta$ has at least one real root, so $\sharp_r \Delta>0$.
\end{proof}

\begin{lemma}
When $PP\in \Gamma_{211}$, then $\sharp_r \Delta>0$.
\end{lemma}
\begin{proof}
When $PP\in \Gamma_{211}$, $p$ has no real zeros, and $p^{'}$ has exactly one real zero $p_0$. To the right of $p_0$, $p^{''}$ has an even number of real zeros. Consequently, (1.2) has a real pole $p_0$, and to its right, (1.2) has an even number of real zeros. Let $z_4$ denote the real zero of $p^{''}$ adjacent to $p_0$ and to the right of $p_0$.

On the real axis to the right of $p_0$, (1.2) has an even number of real zeros. By Lemma 2.3, the interval $(p_0,z_4)$ is a $2q\pi$-degree root locus of (1.2) from a pole to a zero. At the zero $z_4$, the gain $K(z_4)= +\infty$, while at the pole $p_0$, the gain $K(p_0)=0$. Since $(p_0,z_4)$ contains no zeros or poles of (1.2) and $K(x)$ is continuous, $K(x)$ takes all positive real values from $+\infty$ to $0$. Therefore, there exists at least one point in $(p_0,z_4)$ with gain $K_0$, implying that $\Delta$ has at least one root in this interval. Thus, on the real axis to the right of $p_0$, $\Delta$ has at least one root. When $PP\in \Gamma_{211}$, $\sharp_r \Delta>0$.
\end{proof}

There may exist real breakaway points of (1.2) in the interval $(p_0,z_4)$. This implies that $(p_0,z_4)$ contains multiple root loci rather than a single root locus. In the proof of Lemma 3.22, we only provided the proof that the interval $(p_0, z_4)$ is a single root locus. If there exist real breakaway points of (1.2) in the interval $(p_0,z_4)$. In the interval $(p_0,z_4)$, $K(x)$ still attains all positive real values from $0$ to $+\infty$. This doesn't affect the proof and result of Lemma 3.22.

Repeating proof of Lemma 3.22, we can establish Lemma 3.23.

\begin{lemma}
When $PP\in \Gamma_{231}$, then $\sharp_r \Delta>0$.
\end{lemma}
Repeating the proof of Lemma 3.19, we can establish Lemma 3.24.
\begin{lemma}
When $PP\in \Gamma_{2321}\cup\Gamma_{2322}$, let $z_3$ denote the real zero of $p^{''}$ adjacent to $p_0$ and to the left of $p_0$. Let $z_4$ denote the real zero of $p^{''}$ adjacent to $p_0$ and to the right of $p_0$. In the interval $(z_3,p_0)$ and $(p_0,z_4)$, then there are no roots of $\Delta$.
\end{lemma}
When $PP\in \Gamma_{2321}\cup\Gamma_{2322}$, the intervals to the right of $p_0$ are the same as those for $PP\in \Gamma_{2121}\cup\Gamma_{2122}$, and thus we omit their discussion here.

For $PP\in \Gamma_{2321}\cup\Gamma_{2322}$, $p^{''}$ has real zeros to the left of $p_0$. Equation (1.2) has exactly one real pole $p_0$, but multiple real zeros. The interval $(-\infty, z_s)$ may be of $2q\pi$ degree, and these real zeros may form finite intervals of $2q\pi$ degree. Additionally, there exists a real zero $z_3$ of (1.2) adjacent to $p_0$, implying the existence of the interval $(z_3, p_0)$.

The interval $(-\infty, z_s)$ of $2q\pi$ degree may or may not contain standard real breakaway points and constitute the first type of interval. The intervals between real zeros of (1.2) necessarily contain minimum real breakaway points, and the finite intervals of $2q\pi$ degree constitute the second type of interval. The third type consists of intervals $(z_3,p_0)$ extending from its pole $p_0$ to the zeros of (1.2). Therefore, in the proofs of results of rational functions in $\Gamma_{2321}$ and $\Gamma_{2322}$, all these three types of intervals need to be considered.

For the proofs of results concerning rational functions in $\Gamma_{2321}$ and $\Gamma_{2322}$, all three types of intervals must be considered. Since only these three types of intervals exist to the left of $p_0$, the proofs are restricted to these three cases.

\begin{lemma}
When $PP\in \Gamma_{2321}$, then $\sharp_r \Delta=0$.
\end{lemma}
\begin{proof}
When $PP\in \Gamma_{2321}$, $p$ has no real zeros. $p^{'}$ has only one real zero $p_0$. To the right of $p_0$, there is an odd number of real zeros of $p^{''}$. Therefore, (1.2) has a real pole $p_0$. To the right of $p_0$, there is an odd number of real zeros of (1.2). Assume $z_4$ is the real zero of $p^{''}$ adjacent to $p_0$ and located to the right of $p_0$.

In $(z_m,+\infty)$, there is no standard real breakaway point, or all minimum real breakaway points have a gain greater than $K_0$. On the real axis to the right of $p_0$, if there exists a finite interval of $2q\pi$ degree, in all such finite intervals, the gains of all minimum real breakaway points are greater than $K_0$. By Lemma  3.13 and Lemma 3.17, in $(z_m,+\infty)$, there are no roots of the polynomial $\Delta$. Similarly, in all finite intervals of $2q\pi$ degrees, there are no roots of $\Delta$.

By Lemma 3.24, within the interval $(p_0,z_4)$, there cannot be any roots of $\Delta$. Combining the above results, to the right of $p_0$, in $(p_0,+\infty)$, there cannot be any roots of $\Delta$.

To the left of $p_0$, there exist real zeros of $p^{''}$. Assume that $z_3$ is the real zero of $p^{''}$ adjacent to $p_0$ and located to the left of $p_0$. On the real axis to the left of $p_0$, if there exists a finite interval of $2q\pi$ degree, in all such finite intervals, all minimum real breakaway points have a gain greater than $K_0$. Moreover, if there also exists a left-infinite interval $(-\infty, z_s)$ of $2q\pi$ degree,  in $(-\infty, z_s)$, there is no standard real breakaway point, or all minimum real breakaway points have a gain greater than $K_0$. By Lemma 3.17 and Lemma 3.15, in all finite intervals of $2q\pi$ degrees, there are no roots of $\Delta$. In $(-\infty, z_s)$, there are no roots of $\Delta$.  Based on these two results, in the infinite interval $(-\infty,z_3)$ to the left of $z_3$, there are no roots of $\Delta$.

By Lemma 3.24, within the interval $(z_3,p_0)$, there cannot be any roots of $\Delta$. These results confirm that, to the left of $p_0$, there cannot be any roots of $\Delta$.

Combining these results, when $PP\in \Gamma_{2321}$, we conclude that $\sharp_r \Delta=0$.
\end{proof}

\begin{lemma}
When $PP\in \Gamma_{2322}$, then $\sharp_r \Delta>0$.
\end{lemma}
\begin{proof}
When $PP\in \Gamma_{2322}$, $p$ has no real zeros. $p^{'}$ has only one real zero $p_0$. To the right of $p_0$, there is an odd number of real zeros of $p^{''}$. Consequently, (1.2) has a real pole $p_0$. To the right of $p_0$, there is an odd number of real zeros of (1.2). Let $z_4$ be the real zero of $p^{''}$ adjacent to $p_0$ and located to its right.

In $(z_m,+\infty)$, if there exists at least one minimum real breakaway point $b_{m2322}$ such that $K(b_{m2322})\le K_0$, then by Lemma 3.14, there exists at least one root of the polynomial $\Delta$ in $(z_m,+\infty)$. Thus, $\sharp_r \Delta > 0$.

On the real axis to the right of $p_0$, if there exists a finite interval of $2q\pi$ degree, and in at least one such finite interval, there exists at least one minimum real breakaway point $b_{r2322}$ such that $K(b_{r2322})\le K_0$, then by Lemma 3.18, there exists at least one root of $\Delta$ in that $2q\pi$-degree finite interval. Hence, $\sharp_r \Delta>0$.

By Lemma 3.24, within the interval $(p_0,z_4)$, there cannot be any roots of $\Delta$.

On the real axis to the left of $p_0$, there exists a real zero of $p^{''}$. If there exists a left-infinite interval of $2q\pi$ degree to the left of $p_0$, and in this infinite interval, there exists at least one minimum real breakaway point $b_{s2322}$ such that $K(b_{s2322})\le K_0$, then by Lemma 3.16, there exists at least one root of $\Delta$ in this left-infinite $2q\pi$-degree interval. Hence, $\sharp_r \Delta>0$.

Let $z_3$ be the real zero of $p^{''}$ adjacent to $p_0$ and located to its left. If there exists a finite interval of $2q\pi$ degree, and in at least one such finite interval, there exists at least one minimum real breakaway point $b_{l2322}$ such that $K(b_{l2322})\le K_0$, then by Lemma 3.18, there exists at least one root of $\Delta$ in the $2q\pi$-degree finite interval. Thus, $\sharp_r \Delta>0$.

By Lemma 3.24, within the interval $(z_3,p_0)$, there cannot be any roots of $\Delta$.

To summarize all the results above, when $PP\in \Gamma_{2322}$, then $\sharp_r \Delta>0$.
\end{proof}

By combining Lemma 3.4, Lemma 3.6, Lemma 3.20 and Lemma 3.25, we prove Theorem 1.12.

By combining Lemma 3.7, Lemma 3.8, Lemma 3.21, Lemma 3.22, Lemma 3.23 and Lemma 3.26, we establish:

\begin{lemma}
When $PP\in \Gamma_{122}\cup\Gamma_{22}\cup\Gamma_{2122}\cup\Gamma_{211}\cup\Gamma_{231}\cup\Gamma_{2322}$, then $\sharp_r \Delta>0$.
\end{lemma}

By combining Theorem 1 in our paper\cite{Ma} and Lemma 3.27, we prove Theorem 1.11.

\medskip\noindent
\emph{Acknowledgements.} We want to thank Prof.~B.~Shapiro for his advices and help with the preliminary version of the paper. We also gratefully acknowledge the advice from the editors and reviewers.

\medskip
Conflict of interest: On behalf of all the authors, the corresponding author states that there are no conflicts of
interest.

Availability of data and material: No data were used to support this study.

\bibliographystyle{unsrt}

\end{document}